\documentclass[11pt]{amsart}

\usepackage{color}
\usepackage{amsbsy}

 \usepackage{amssymb,amsmath,amstext,mathrsfs}
\usepackage{enumerate}

\makeatletter
\def\verbatim{\interlinepenalty\@M \@verbatim
  \leftskip\@totalleftmargin\advance\leftskip2pc
  \frenchspacing\@vobeyspaces \@xverbatim}
\makeatother
\newtheorem{thm}{Theorem}[section]
\newtheorem{cor}[thm]{Corollary}
\newtheorem{lem}[thm]{Lemma}
\newtheorem{prop}[thm]{Proposition}
\newtheorem{claim}[thm]{Claim}

\theoremstyle{definition}
\newtheorem{defn}{Definition}[section]

\theoremstyle{remark}
\newtheorem{rem}{Remark}[section]

\numberwithin{equation}{section}

\newtheorem{ex}[thm]{Example}

\newcommand{\va}{\left|}
\newcommand{\vb}{\right|}
\newcommand{\vd}{\right\| }
\newcommand{\vc}{\left\| }

\newcommand{\pl}{\left(}
\newcommand{\pr}{\right)}
\newcommand{\ql}{\left[}
\newcommand{\qr}{\right]}
\newcommand{\tl}{\left\{}
\newcommand{\tr}{\right\}}

\newcommand{\ra}{\rightarrow}

\newcommand{\de}{\mathbb{D}}
\newcommand{\ka}{\mathbb{K}}
\newcommand{\en}{\mathbb{N}}

\newcommand{\er}{\mathbb{R}}
\newcommand{\ce}{\mathbb{C}}
\newcommand{\qu}{\mathbb{Q}}

\newcommand{\hac}{H^{\infty}}

\newcommand{\ma}{\mathfrak{M}}

\newcommand{\ze}{\mathbf{z}}

\newcommand{\we}{\mathbf{w}}
\newcommand{\ve}{\mathbf{v}}

\newcommand{\mil}{\ma_0}
\newcommand{\mim}{\ma'_0}

\newcommand{\muz}{\ma_{\ze, \ef}}
\newcommand{\mak}{\ma_1}
\newcommand{\maz}{\ma_{\ze}}

\newcommand{\mut}{A_{\ze , \ef}^{\kai}}
\newcommand{\myt}{Z_{\ze , \ef}^{\kai}}

\newcommand{\ef}{\mathfrak{u}}

\newcommand{\af}{\mathfrak{v}}

\newcommand{\com}{\mathtt{Comp}}
\newcommand{\cer}{\mathtt{Z}}

\newcommand{\vp}{\zeta}

\newcommand{\eli}{\mathbf{l}}

\newcommand{\emi}{\mathbf{m}}

\newcommand{\kai}{\mathbf{k}}

\DeclareMathOperator{\ca}{card}

\DeclareMathOperator{\cl}{\mathrm{cl}_{\ma}}
\DeclareMathOperator{\cn}{\mathrm{cl}_{\beta \en}}
\DeclareMathOperator{\cw}{\mathrm{cl}_{\mathit{A}}}

\title[Nonmaximal ideals]{Nonmaximal ideals and the Berkovich space of the algebra of bounded analytic functions}

\author{Jes\'us Araujo}
\address{Departamento de Matem\'aticas,
Estad\'{\i}stica y Computaci\'on\\ Universidad de Cantabria\\
Facultad de Ciencias\\ Avda.
de los Castros, s. n.\\ E-39071 Santander, Spain}

\email{araujoj@unican.es}
\thanks{Research
partially supported by
the Spanish Ministry of Science and Education (Grant MTM2011-23118).}

\date{}
\begin{document}

\begin{abstract}
We prove that the  Berkovich space   of the  algebra of bounded analytic functions on the open unit disk of an  algebraically closed nonarchimedean field   contains multiplicative seminorms that are not norms and whose kernel is not  a  maximal ideal. We also prove that in general these seminorms are not univocally determined by their kernels, and provide a method for obtaining families of different seminorms sharing the same kernel. On the other hand, we prove that there are also kernels that cannot be obtained by that method. The relation with the Berkovich space of the Tate algebra is also given.
\end{abstract}

\maketitle

\section{Introduction}

Throughout $\ka$ is an algebraically closed field complete with respect to a (nontrivial) nonarchimedean absolute value $\va \cdot \vb$ and $\hac$  denotes the space of ($\ka$-valued) bounded analytic functions on the open   disk $\de := \tl z \in \ka   :    \va z \vb <1 \tr$, that is, the space   of  bounded power series  on $\de$. When endowed with the Gauss norm (which coincides with the sup  norm $\vc \cdot \vd$), the space $\hac$ becomes a Banach algebra. We  remark that, given a nonzero $f (z) = \sum_0^{\infty} a_n z^n \in \hac$, the value $$\vc f \vd = \sup_{n \ge 0} \va a_n \vb = \sup_{z \in \de} \va f(z) \vb$$ {\em does not} necessarily belong to the value group $\va \ka^{\times} \vb := \tl \va z \vb : z \in \ka\setminus \tl 0 \tr \tr$.

 A remarkable difference with respect to the complex case is that in a Banach algebra over $\ka$ there can be maximal ideals that are not the kernel of any multiplicative linear functional. For this reason, the classical definition of spectrum (or maximal ideal space)  of a complex Banach algebra does not carry over to the ultrametric setting.
 Nevertheless, the standard definition of Berkovich space (or multiplicative spectrum)
    yields the usual spectrum when adapted to the complex context 
 (see Definition~\ref{hurrodemarta}  and Remark~\ref{gafarrota}).

Not much is known about the Berkovich space $\ma$ of $\hac$.
Points in $\ma$ are {\em seminorms}, and
 theoretically  they can be divided into four types, namely: 
\begin{enumerate}[I.]
\item Points whose kernel is a maximal ideal of codimension $1$,
\item Points whose kernel is a maximal ideal of codimension different from $1$,
\item Points whose kernel is trivial, that is, equal to $\tl 0 \tr$.
\item Points whose kernel is a nonzero nonmaximal  {\em prime}  ideal.
\end{enumerate}

Points of type I can be identified with those in $\de$ (see \cite{EM}),
as each of them
is the  absolute value evaluation $\delta_z$ at a point $z$ of $\de$ (that is,  $\delta_z (f) = \va f(z) \vb$  for every $f \in \hac$).

Points of type II can be obtained by the use of ultrafilters and, in particular, {\em regular} sequences
(a sequence $(z_n)$ in $\de$  is said to be regular if $\inf_{n \in \en} \prod_{m \neq n} \va z_n - z_m \vb >0$). The
 key point in studying regular sequences consists of identifying each of them with a bounded sequence in $\ka$ via the map
 $\mathbf{i} : \hac \ra \ell^{\infty}$, $f \in \hac \mapsto (f(z_n)) \in \ell^{\infty}$. Given a regular sequence
$(z_n)$, every maximal ideal containing the ideal $\mathfrak{I}$ of all functions $f \in \hac$
 vanishing at every $z_n$ can be identified
with an ultrafilter in $\en$, that is, a point in   the Stone-\v{C}ech compactification $\beta \en$ of $\en$ (see \cite[Corollary 4.7]{vdP}). Thus, given a regular sequence $\ze = (z_n)$ and a nonprincipal ultrafilter $\ef$ in $\en$ (that is, a point $\ef \in \beta \en \setminus \en$), the seminorm $$\delta_{\ze, \ef} := \lim_{\ef} \delta_{z_n} $$is a point of type II.
In this paper, we say that a sequence $(z_n)$ in $\de$ is regular with respect to a nonprincipal ultrafilter $\ef$ in $\en$ if there exists $C \in \ef$ such that $(z_n)_{n \in C}$ is regular, that is,
$$\inf_{n \in C} \prod_{\substack{m \in C \\ m \neq n}} \va z_n - z_m \vb >0.$$

Points of type III are obviously given by multiplicative norms.
 The simplest case of a multiplicative norm is of the form $\zeta_{D}$, for any nontrivial disk $D$   contained in $\de$, where $$\zeta_D (f) := \sup_{z \in D} \va f(z) \vb$$ for all $f \in \hac$.

Our goal in this paper is to prove that the set of points of type IV is nonempty, and to study some of its features.
Note
that the existence of a nonzero nonmaximal closed prime  ideal does not necessarily imply the existence of points of type IV.
 The
question of the existence of
 such an ideal in $\hac$, raised in \cite{vdP},  remained unknown for many years, until it was finally  solved (in the positive) in \cite{E3}. On the other hand, note that, from the proof of \cite[Proposition 1.2.3]{B}, the algebra $\ell^{\infty}$ contains no nonzero nonmaximal prime ideals. In our case,    $\hac$ and $\ell^{\infty}$ are isometrically isomorphic   as   Banach spaces, but the product in $\hac$ determines a product in $\ell^{\infty}$ different from the usual one, allowing the existence of seminorms that are not norms and have  a nonmaximal ideal as a kernel.

\begin{defn}\label{hurrodemarta}
Let $A$ be a unital commutative Banach algebra  over $\ka$. A
map $\varphi : A \ra [0, + \infty)$ is a continuous multiplicative
{\em ring} seminorm on $A$ if the following conditions hold:
\begin{enumerate}
\item $\varphi (0_A ) = 0$ and $\varphi (1_A) =1$.
\item $\varphi  (ab) = \varphi (a) \ \varphi  (b)$ for all $a, b \in A$.
\item $\varphi (a + b) \le  \varphi (a) + \varphi (b)  $ for all $a, b \in A$.
\item $\varphi  (a) \le \vc a \vd $ for all $a \in A$.
\end{enumerate}
\end{defn}

\begin{rem}\label{gafarrota}
We assume that $\vc 1_A \vd =1$. It is straightforward to show (see for instance  \cite[Lemma 1.7]{E1}) that every continuous multiplicative ring seminorm
is also an {\em ultrametric algebra} seminorm on $A$, that is,  it further satisfies:
\begin{enumerate}
\setcounter{enumi}{4}
\item $\varphi (\lambda a ) = \va \lambda \vb \varphi (a)$ for all $\lambda \in \ka$ and $a \in A$.
\item $\varphi (a + b) \le \max \tl \varphi (a) , \varphi (b) \tr$ for all $a, b \in A$.
\end{enumerate} 
\end{rem}

The Berkovich space (or multiplicative spectrum) $\mathscr{M} (A)$ of $A$ is the set of all  continuous multiplicative (in any of the equivalent senses of Definition~\ref{hurrodemarta} and Remark~\ref{gafarrota})  seminorms endowed with the topology of
simple convergence, that is,
a net $\pl \zeta_{\lambda } \pr_{\lambda \in \Lambda}$ in $\mathscr{M}(A)$ converges to $\zeta_0 \in \mathscr{M} (A)$ if $\pl \zeta_{\lambda } (a) \pr_{\lambda \in \Lambda}$ converges to $\zeta_0 (a)$ for all $a \in A$.
It is well known that $\mathscr{M}(A)$ is Hausdorff and compact (see for instance \cite[Theorem 1.2.1]{B} or \cite[Theorem 1.11]{E1}).
 Indeed,
  the multiplicative spectrum of some algebras
  is a compactification of $\de$ (see \cite{B,BR,E1,E2,Gb,S}). Nevertheless, in our case, it is unknown if $\de$ is dense   in  $\ma  = \mathscr{M} \pl \hac \pr$, which is a nonarchimedean version of the Corona problem
(a related problem was solved in \cite{vdP}). In fact, what is now known is that  $\de$ is   dense in the subset of all seminorms whose kernel is a maximal ideal (see \cite{E4}).

It is easy to check that the kernel
$\ker \vp := \tl f \in A : \vp (f) = 0 \tr$
of every element $\vp \in \mathscr{M} (A)$ is a closed prime ideal of $A$. When we say that a seminorm has {\em maximal kernel} or {\em nonzero nonmaximal  kernel}, we mean that its kernel is a {\em maximal ideal} or a {\em nonzero nonmaximal  ideal}, respectively.

 We see that
if $D$ is a (closed or open) disk, then  $\vp_D$ belongs to $\ma$.
 Also, since  $\va \ka^{\times} \vb$ is dense in $\er^+$, $\vp_{D^+ (z, r)} =   \vp_{D^- (z, r)}$  for $z \in \de$ and $r \in (0,1)$ (where $D^+ (z, r)$ and $D^- (z,r)$ are the closed and open disks with center $z$ and radius $r$, respectively).

Recall that, given $f \in \hac$ and $z_0 \in \de$, $f$ can be written by $f(z) = \sum_{n=1} a_n (z-z_0)^n$ for every $z \in \de$ (see for instance \cite[Theorem 25.1]{Sch}), and that $z_0$ is a zero of $f$ of multiplicity $m \ge 1$ if there is $g \in \hac$ with $g(z_0) \neq 0$ such that $f(z) = (z-z_0)^m g(z)$ for all $z$. For $E \subset \de$, we denote by $\cer (f,  E)$ the number of zeros of $f$ in $E$
(by this we will always mean {\em taking into account multiplicities}).

For $r>0$,  $C (0,r)$ will be the set of all $z $ with  $\va z \vb =r$.
If $D^+ (z, r) \subset C(0, \va z \vb)$ and $w_1, \ldots, w_n$ are the zeros of $f \in \hac$ with absolute value $\va z \vb$, then we define
$$\xi_{D^+ (z, r)} (f) :=
 \left\{ \begin{array}{ll} r^{\cer (f, D^+ (z, r))} \prod_{\va z- w_i \vb >r} \va z- w_i \vb & \mbox{if }   \cer (f, C(0, \va z \vb )) \neq 0,
\\
1 & \mbox{if }   \cer (f, C(0, \va z \vb )) = 0,
\end{array} \right.
$$
where we understand that $\prod_{\va z- w_i \vb >r} \va z- w_i \vb =1$ if  $\va z- w_i \vb \le r $ for all $i=1, \ldots, n$.

\medskip

In this paper we mainly study the set $\mil$ of all seminorms of the form $\varphi := \lim_{\ef} \zeta_{D^+ \pl z_n , r_n \pr}$, where $\ef$ is any nonprincipal ultrafilter in $\en$, $(z_n)$ is any  sequence in $\de$ with $\lim_{n \ra \infty} \va z_n \vb =1$, and $(r_n)$ is any sequence in $(0,1)$.
Obviously, in many cases $\varphi := \lim_{\ef} \zeta_{D^+ \pl z_n , r_n \pr} = \vc \cdot \vd$. This happens for instance when the set of all $n \in \en$ such that $\va z_n \vb \le r_n$ belongs to $\ef$. But even in this case we can also write $\vc \cdot \vd = \lim_{\ef} \zeta_{D^+ \pl z_n , \va z_n \vb^2 \pr} $, so we can assume that $r_n < \va z_n \vb$ for all $n$. It is clear that, if $\lim_{\ef} r_n >0$, then there exist $r \in (0,1)$ and a sequence $\kai = (k_n)$ in $\en$ (not necessarily unique) such that $0 < r = \lim_{\ef} {r_n}^{k_n} < 1$.
We will see in Corollary~\ref{pernia}  that   $$\varphi = \varphi_{\ze, \ef}^{\kai , r} := \lim_{\ef} \zeta_{D^+ \pl z_n , \sqrt[k_n]{r} \pr}.$$   This means that, when $\lim_{\ef} r_n >0$, we can restrict ourselves to seminorms of the special form $\varphi_{\ze, \ef}^{\kai , r}$.
On the other hand,
 it is very easy to see that, if $\lim_{\ef} r_n =0$, then $\lim_{\ef} \zeta_{D^+ \pl z_n , r_n \pr } = \delta_{\ze, \ef} := \lim_{\ef} \delta_{z_n}$. We
 also prove that, in fact, all points in $\mil$ can be written in the form $\delta_{\ze, \ef}$ (see Theorem~\ref{nadabouche}).

We can say more. Given    $\varphi   \in \mil$, there exist a sequence $(w_n )$   in $\de$ with $\lim_{n \ra \infty} \va w_n \vb =1$ and a sequence $(s_n) $    in $(0,1)$ such that  the disks $D^+ \pl w_n , s_n \pr$ are pairwise disjoint and  $\varphi = \lim_{\ef} \zeta_{D^+ \pl w_n , s_n \pr}$ (see Corollaries~\ref{byemarcal} and~\ref{puigdemont}).

We also deal here with two subsets of $\ma_0$: $\mim$ and $\mak$. The
set $\mim$  consists of all the limits of the above form $\lim_{\ef} \zeta_{D^+ \pl z_n , r_n \pr}$, where   $(z_n)$ is   {\em regular with respect to} $\ef$ and all the disks $D^+ \pl z_n , r_n \pr$, $n \in C$, are pairwise disjoint for some $C \in \ef$.
If we drop the requirement that $(z_n)$  be regular with respect to $\ef$, then the results we obtain are quite different (see Proposition~\ref{gertrudisrivota}; see also Corollary~\ref{puigdemont}).

As for the second set, $\mak$, it has the remarkable property that no seminorm in it is determined by its kernel, that is, there are many other seminorms having the same kernel. For the description of $\mak$,   we generalize the notion of regular sequence as follows:
Given a sequence $\ze = (z_n)$ in $\de$
  and a nonprincipal ultrafilter $\ef$ in $\en$, we denote by $\com_{\ef}  (\ze)$ the set of all sequences $\kai = (k_n)$ in $\en$ 
for which
there exists $C_{\kai} \in \ef$ such that $$\inf_{n \in C_{\kai}} \prod_{\substack{m \in C_{\kai} \\ m \neq n}} \va z_n - z_m \vb^{k_m} >0 .$$

Now, for a nonprincipal ultrafilter $\ef$ of $\en$, $\kai \in \com_{\ef}  (\ze)$ and $r \in (0,1)$,
we set $\zeta_{\ze, \ef}^{\kai, r} := \varphi_{\ze, \ef}^{\kai, r}$, that is, $$\zeta_{\ze, \ef}^{\kai, r}   := \lim_{\ef} \zeta_{D^+ \pl z_n , \sqrt[k_n]{r} \pr} ,$$ and
$$\pl \zeta_{\ze, \ef}^{\kai, 0} , \zeta_{\ze, \ef}^{\kai, 1} \pr := \tl \zeta_{\ze, \ef}^{\kai, r}   : r \in (0,1) \tr .$$
We  put, for $\ze$ and $\ef$ fixed,
 $ \muz :=  \bigcup_{\kai  \in \com_{\ef}  (\ze)} \pl \zeta_{\ze, \ef}^{\kai, 0} , \zeta_{\ze, \ef}^{\kai, 1} \pr    $,
and more in general  $\maz := \bigcup_{\ef \in \beta \en \setminus \en} \muz $.
Finally, we set $\mak := \bigcup_{\ze} \maz$.

Note that, in principle, a seminorm  $\varphi_{\ze, \ef}^{\kai, r} \in \mim$ cannot be written as  $\zeta_{\ze, \ef}^{\kai, r}$ because $\kai$ does not necessarily belong to $\com_{\ef} (\ze)$ (nevertheless,  in general it does, as can be seen in  Theorem~\ref{juntafaculta}).
On the other hand, $\mak$ is indeed a subset of $\mim$ (see Remark~\ref{nonada}). But, of course, the fact that a seminorm $\zeta_{\ze, \ef}^{\kai, r} \in \mak$ belongs to $\mim$ does not necessarily imply that there exists $C \in \ef$ such that all disks $D^+ \pl z_n, \sqrt[k_n]{r}  \pr$ are pairwise disjoint for $n \in C$. Nevertheless, we have the following remark that will be used later.
\begin{rem}\label{honble}
 If there exists $C \in \ef$ with $M := \inf_{n \in C} \prod_{\substack{m \in C \\ m \neq n}} \va z_n - z_m \vb^{k_m} >0$ and  $0 < r_0 <M$, then   all the  disks $D^+ \pl z_n , \sqrt[k_n]{r_0} \pr$, $n \in C$, are pairwise disjoint.
\end{rem}

By $\mathbf{1}$, we denote the sequence constantly equal to $1$. In general, $\kai$, $\eli$, $\emi$ are used, respectively, for sequences $(k_n)$, $(l_n)$ and $(m_n)$ in $\en$. Also $\ze$, $\we$, and $\ve$ denote,   respectively,  sequences $(z_n)$, $(w_n)$ and $(v_n)$ in $\de$.

As usual, given a topological space $A$ and a subset $B$ of $A$, $\cw B$ denotes the closure of $B$ in $A$.

\medskip

The paper is organized as follows. In Section~\ref{noporu} we state the main results.  In Section~\ref{complauton}, we give some technical results that are used through the paper. In Section~\ref{mieko}, we show that the Berkovich space of the Tate algebra $T_1$ (without one point) can be homeomorphically embedded as an open subset of $\ma$ (Theorem~\ref{casitadechocolate}). In Section~\ref{gatopo}, we study the existence of bounded analytic functions with a prescribed number of zeros, paying attention to their norms.
In Section~\ref{nomarkator}, we study how the same seminorm can be expressed in different forms, and we prove in particular Theorem~\ref{nadabouche}. Section~\ref{davai} is devoted to proving most of the  results stated in Section~\ref{complauton} (and some others concerning $\mak$). The proof of Theorem~\ref{zerolari}  is provided in Section~\ref{ceropombo}, along with a description of some special seminorms.

\section{Main results}\label{noporu}

 \begin{thm}\label{liberban}
Let $\varphi \in \mil$ have nonzero kernel. Given $f \in \ker \varphi$ with $f \neq 0$    and
$r \in \pl 0, \vc f \vd \pr$,  there exists $\psi \in \mim$ with
nonzero nonmaximal   kernel   such that $\varphi \le \psi$  and $\psi (f) =r$.
\end{thm}

We deduce that  $\ker \psi \subsetneq \ker \varphi$, and consequently   there exists an infinite strictly decreasing chain of kernels.
 In particular all kernels of seminorms $\delta_{\ze, \ef}$, with $\ze$ regular with respect to $\ef$, contain nontrivial kernels.

\begin{thm}\label{kubzdelamagdalena}
Let $\ze$ be regular with respect to a nonprincipal ultrafilter $\ef$ in $\en$. Then there exists a linearly ordered compact and connected set $A_{\ze}^{\ef}   \subset \ma$ with $\delta_{\ze, \ef} = \min A_{\ze}^{\ef}$ and $\vc \  \vd = \max A_{\ze}^{\ef}$
such that $\ker \varphi$ is nonzero and nonmaximal for all
$\varphi \in A_{\ze}^{\ef} \setminus \tl \delta_{\ze, \ef} , \vc \  \vd\tr$.
\end{thm}

Points in $\mil$ can in fact be written in the form $\delta_{\we, \af} := \lim_{\af} \delta_{w_m}$, where $\we$ may be {\em not} regular with respect to $\af$.

\begin{thm}\label{nadabouche}
 $\mil = \tl \delta_{\we, \af} : \lim_{n \ra \infty} \va w_n \vb =1,  \af \in \beta \en \setminus \en \tr$.
\end{thm}

\begin{thm}\label{thesidubar}
Let
$\ze$ be a regular sequence with respect to $\ef \in \beta \en \setminus \en$. Then, for each $\kai \in  \com_{\ef}  (\ze) $, the maps $\zeta_{\ze, \ef}^{\kai, 0} := \lim_{r \ra 0}   \zeta_{\ze, \ef}^{\kai, r}$ and $ \zeta_{\ze, \ef}^{\kai, 1}:= \lim_{r \ra 1} \zeta_{\ze, \ef}^{\kai, r}$ exist and
belong to $\ma$, and
$$\cl \pl \zeta_{\ze, \ef}^{\kai, 0} , \zeta_{\ze, \ef}^{\kai, 1} \pr =     \pl \zeta_{\ze, \ef}^{\kai, 0} , \zeta_{\ze, \ef}^{\kai, 1} \pr \ \cup \tl \zeta_{\ze, \ef}^{\kai, 0}  ,  \zeta_{\ze, \ef}^{\kai, 1} \tr .$$
Moreover  $\cl \pl \zeta_{\ze, \ef}^{\kai, 0} , \zeta_{\ze, \ef}^{\kai, 1} \pr $ is homeomorphic to the interval $[0,1]$, through a homeomorphism sending $\pl \zeta_{\ze, \ef}^{\kai, 0} , \zeta_{\ze, \ef}^{\kai, 1} \pr$ onto $(0,1)$.
\end{thm}

In fact, seminorms $\zeta_{\ze, \ef}^{\kai, 0}$ and $\zeta_{\ze, \ef}^{\kai, 1}$  belong to $\mil$ (see Proposition~\ref{tobien}).

The following result  says that many seminorms share the same nonzero nonmaximal kernels (see also Corollary~\ref{lz}).

\begin{cor}\label{pase}
Let
$\ze$ be a regular sequence with respect to $\ef \in \beta \en \setminus \en$. Then, for each $\kai \in  \com_{\ef}  (\ze) $, all seminorms in $\pl \zeta_{\ze, \ef}^{\kai, 0} , \zeta_{\ze, \ef}^{\kai, 1} \qr := \pl \zeta_{\ze, \ef}^{\kai, 0} , \zeta_{\ze, \ef}^{\kai, 1} \pr \cup \tl \zeta_{\ze, \ef}^{\kai, 1} \tr$ have the same kernel.
\end{cor}

In view of Corollary~\ref{pase}, we can consider kernels of seminorms in $\muz$ given by different sequences $\kai$ and $\eli$. It is very easy to deduce that they coincide when $\lim_{\ef} l_n /k_n \in (0, + \infty)$. In any other case, we have the following corollary.

\begin{cor}\label{zerote}
Let
$\ze$ be a regular sequence with respect to $\ef \in \beta \en \setminus \en$. Let $\kai , \eli \in  \com_{\ef}  (\ze) $. If $\lim_{\ef} l_n / k_n = 0$, then $\ker \zeta_{\ze, \ef}^{\kai, 1} \subsetneq \ker \zeta_{\ze, \ef}^{\eli, 1} $.
\end{cor}

\begin{cor}\label{lluc}
The kernel of every point in $\mak$ is nonzero and nonmaximal.
\end{cor}

We easily deduce that   $\ker \zeta_{\ze, \ef}^{\kai, 1}$ is always nonzero and nonmaximal, and that $\ker \zeta_{\ze, \ef}^{\kai, 0}$ is   nonzero.
Moreover, if $\lim_{\ef} k_n < + \infty$, then $\zeta_{\ze, \ef}^{\kai, 0} = \delta_{\ze, \ef}$, so its kernel is maximal. Now, we see that
the converse also holds.

\begin{cor}\label{lucia}
Let
$\ze$ be a regular sequence with respect to $\ef \in \beta \en \setminus \en$. Then, for each $\kai \in  \com_{\ef}  (\ze) $,
 $\ker \zeta_{\ze, \ef}^{\kai, 0}$ is
nonmaximal   if and only if
$\lim_{\ef} k_n = + \infty$.
\end{cor}

Next, if $\varphi_{\ze, \ef}^{\kai, r}, \varphi_{\ze, \ef}^{\eli, s}   \in   \mim$ do not belong to $\mak$, then $\varphi_{\ze, \ef}^{\kai, r} = \varphi_{\ze, \ef}^{\eli, s} $. That is, all points
in $ \mim$ (with nonmaximal kernel) belong to $\mak$ but at most one:

\begin{thm}\label{juntafaculta}
Given
$\varphi = \varphi_{\ze, \ef}^{\kai, r}   \in   \mim$,   either   $\varphi   \in \muz$  or
 $$\varphi = \sup_{\emi \in \com_{\ef} (\ze)} \zeta_{\ze, \ef}^{\emi,1}.$$
 \end{thm}

\begin{cor}\label{sussosoplador}
Let $ \varphi = \varphi_{\ze, \ef}^{\kai, r} \in \mim$, where $(\va z_n \vb)$ is strictly increasing. Then  either $\varphi \in \mak$ or $\varphi = \vc \ \vd$.
\end{cor}

The next result says that there are some kernels that cannot be obtained through seminorms in $\mak$. In fact we  see that, among the seminorms of the form $\zeta^{\kai, r}_{\ze, \ef}$, those with $r =0$ are the only ones that are characterized by their kernel. This  should be compared with Corollary~\ref{lz}. Example~\ref{nicolas} shows that the statement cannot be generalized to other seminorms defined in a similar way.

\begin{thm}\label{zerolari}
Let
$\ze$ be a regular sequence with respect to $\ef \in \beta \en \setminus \en$.  Given
  $\kai \in  \com_{\ef}  (\ze) $,
     $$\ker \zeta_{\ze, \ef}^{\kai, 0} \neq \ker \varphi$$ for every $\varphi \in \mak$.

     Moreover, for $ \af \in \beta \en \setminus \en$ and a regular sequence $\we$ with respect to $\af$, if $ \eli   \in
 \com_{\af}  (\we)$ and $\ker \zeta^{\kai, 0}_{\ze, \ef} = \ker \zeta^{\eli, 0}_{\we, \af}$, then $\zeta^{\kai, 0}_{\ze, \ef} =   \zeta^{\eli, 0}_{\we, \af}$.
\end{thm}

We finish our list of main results with a theorem linking the Berkovich space  of the Tate algebra $T_1$  with $\ma$. Recall that $T_1$ is the Banach algebra of analytic functions on the closed unit disk $D^+ \pl 0, 1 \pr$, that is, the space of all power series with coefficients in $\ka$ converging on $D^+ \pl 0, 1 \pr$. It coincides with the subspace of $\hac$ consisting of all power series $\sum_{n=0}^{\infty} a_n z^n$ with $\lim_{n \ra \infty} \va a_n \vb =0$, and contains the polynomial algebra $\ka [z]$ as a dense subset. 

The Berkovich space  $\mathscr{M} \pl T_1   \pr $  is well known (see for instance \cite[1.4.4]{B}). Each $\varphi \in \mathscr{M} (T_1)$ can be written in terms of (a limit of) seminorms $\zeta_{D^+ \pl a, r \pr}$, in such a way that there is a natural extension of each $\varphi$ to a   $\mathbf{i} \pl \varphi \pr \in \ma$ defined in the same terms. We put $\mathscr{M}^* := \mathscr{M} \pl  T_1 \pr \setminus \tl \vc \ \vd \tr  $.

\begin{thm}\label{casitadechocolate}
The canonical map
$$\mathbf{i}: \mathscr{M}^* \ra \mathbf{i} \pl \mathscr{M}^* \pr \subset \ma$$
 is a homeomorphism. Moreover $\mathbf{i} \pl \mathscr{M}^* \pr  $ is open in $\ma$, and $\mathscr{M} \pl  T_1 \pr$ is homeomorphic to a quotient of $\ma$.
\end{thm}

\section{Some technical results}\label{complauton}

We begin this section by giving some well known results concerning the zeros of analytic functions. Suppose that
$f (z) =   1 + \sum_{n=1}^{\infty} a_n z^n \in \hac$. For each $r \in [0, 1)$, let
$M_r (f) := \max_{n \ge 0} \va a_n \vb r^n$.
We say that $r \in (0,1)$ is a critical radius for $f$ if there are at least two distinct indices $m, k$ such that
$$M_r (f) = \va a_m \vb r^m = \va a_k \vb r^k.$$
It turns out that $r$ is a critical radius for $f$ if and only if $C(0,r)$ contains a zero of $f$. Indeed, the number of zeros (taking into  account multiplicities)  located in $C(0,r)$ coincides with the number
$$\cer (f, C(0,r)) = \nu_r (f) - \mu_r (f)$$
where $\nu_r (f)$ and $\mu_r (f)$ are defined, respectively, as the greatest and the smallest $n$ such that $\va a_n \vb r^n = M_r (f)$ (see for instance \cite[Section 2.2, Theorem 1]{R} for a proof when  $\ka$  is an algebraically closed extension of $\qu_p$ but valid also for our $\ka$).
It is clear from the definition that, if $r <s$, then $\nu_r (f) \le \mu_s (f)$. In fact, the
critical radii form an increasing (finite or infinite) sequence $(R_n)$
satisfying  $\mu_{R_n} (f) =  \nu_{R_{n-1}} (f) $ for all $n \ge 2$ that, when infinite, has $1$ as its only accumulation point.

Hence, if $r \in (0,1)$ is not a critical radius, then there exists only one $n_r \in \en$ with  $\va a_{n_r} \vb r^{n_r} = M_r (f)$
and $\va f (z) \vb = \va a_{n_r} \vb r^{n_r}$ for all $z$ with $\va z \vb =r$. It turns out that $n_r = \nu_{R_i} (f)$, where $R_i$ is the greatest critical radius strictly less than $r$, if there is any, and $n_r  =\mu_{R_1} (f) =0$ otherwise.

On the other hand, writing $\nu_n =\nu_{R_n} (f)$ for short,    we see that $\va f (0) \vb = 1 = \va a_{\nu_1}  \vb {R_1}^{\nu_1}$ and
$ \va a_{\nu_1} \vb = 1/{R_1}^{\nu_1}$.
Also,   $\va a_{\nu_1}  \vb {R_2}^{\nu_1} = \va a_{\nu_2}  \vb {R_2}^{\nu_2} $, giving
$\va a_{\nu_2} \vb = 1/\pl {R_1}^{\nu_1} {R_2}^{\nu_2 - \nu_1} \pr = 1 /\prod_{i=1}^2 {R_i}^{\cer (f, C(0, R_i))}$.
For all $n$, this process leads to
 $ \va a_{\nu_n} \vb = 1/\prod_{i=1}^n {R_i}^{\cer (f, C(0, R_i))} $.
We finally remark that $$\vc f \vd = \sup_{n} \va a_{\nu_n} \vb =  \frac{1}{ \prod_{i=1}^{\infty} {R_i}^{\cer (f, C(0, R_i))}} .$$
\medskip

We continue with the results of this section. The proof of the following lemma is easy.

 \begin{lem}\label{khatxarel}
Suppose that  $f \in \hac$ has exactly $k$ zeros $w_1, \ldots, w_k$ of absolute value $R \in (0,1)$. Then
$f (z) = g  (z) \prod_{i=1}^k (z - w_i)$, where $g \in \hac$ has no zeros of absolute value $R$. Also,
$\vc f \vd = \vc g \vd$.
\end{lem}

  \begin{lem}\label{sinaltaboces}
Let $f  \in \hac$ be such that $f(0) =1$, and suppose that its
  critical radii   are $R_1 < R_2 < \cdots <1$.
Suppose also that for each $i \in \en$,
 $f$ has exactly $m_i$ zeros $w_1^i, \ldots, w_{m_i}^i$ in  $C(0, R_i)$.
 Then, given
 $z \in \de$ with $\va z \vb = R_k $,
$$\va f(z ) \vb =  \frac{{R_k}^{m_1 + \cdots + m_{k-1}}  \prod_{j=1}^{m_k} \va z - w_j^k \vb}{\prod_{i=1}^{k} {R_i}^{m_i}} . $$
Similarly, if $R_k < R:=  \va z \vb <  R_{k+1} $, then
$$\va f(z ) \vb =  \frac{{R}^{m_1 + \cdots + m_k}}{\prod_{i=1}^{k} {R_i}^{m_i}} . $$
\end{lem}

\begin{proof}
By Lemma~\ref{khatxarel}, $f (z) = g (z) \prod_{i=1}^k \prod_{j=1}^{m_i} \pl z-   w_j^i \pr$, where $g \in \hac$ has $m_i$ zeros in each
 $C(0, R_i)$ for every $i > k$, and no other zeros. This implies that the critical radii of $g$
 are the $R_i$ for $i > k$ and that $\va g (z) \vb$ is constantly equal to $\va g(0) \vb$ on $D^- \pl 0, R_{k+1} \pr$,
 that is, when $\va z \vb < R_{k+1}$, $$\va g (z) \vb = \va g(0) \vb  = 1/{R_1}^{m_1} \cdots {R_k}^{m_k}.$$
Now, the result follows easily.
\end{proof}

\begin{cor}\label{angiosperma}
Suppose that $f \in \hac$ has no zeros in $D^- \pl z_0, r \pr$, where $0 < r \le \va z_0 \vb$. Then $\va f(z ) \vb = \va f \pl z_0 \pr \vb$
for every $z \in D^- \pl z_0, r \pr$, and $\zeta_{D^+ \pl z_0, r \pr} \pl f \pr = \va f \pl z_0 \pr \vb$.
\end{cor}

Corollary~\ref{conillosis}  will be very useful.

\begin{cor}\label{conillosis}
Let $\ze$ be a sequence in $\de$ with $ ( \va z_n \vb) $ increasing and converging to $1$, and let $(r_n)$ be a sequence in $(0,1)$ with $D^+ (z_n ,r_n) \subset C(0, \va z_n \vb)$ for all $n$. Given   a nonprincipal ultrafilter $\ef$ in $\en$,
$$\lim_{\ef} \zeta_{D^+ (z_n ,r_n)} (f) = \vc f \vd  \lim_{\ef} \xi_{D^+ (z_n ,r_n)} (f) $$
for every $f \in \hac$.
\end{cor}

\begin{proof}
Since each $\xi_{D^+ (z_n ,r_n)}$ is multiplicative,  it is enough to prove it for $f \in \hac$  with $f(0) =1$. Also, the result is obvious if  $f$ has a finite number of zeros in $\de$, so we assume that    the sequence  $(w_k)$ of its zeros satisfies that $(\va w_k \vb)$ is increasing and convergent to $1$.

For each $n \in \en$, take $k_n$ as the  largest $k$ with $\va w_k \vb \le \va z_n \vb$. If $\va w_{k_n} \vb < \va z_n \vb$, then
$\xi_{D^+ (z_n ,r_n)} (f) = 1$ and, by Lemma~\ref{sinaltaboces},
$$   \zeta_{D^+ (z_n ,r_n)} (f) = \frac{\va z_n \vb^{k_n}}{\prod_{i=1}^{k_n} \va w_i \vb} = \frac{\va z_n \vb^{k_n}}{\prod_{i=1}^{k_n} \va w_i \vb} \ \xi_{D^+ (z_n ,r_n)} (f) .$$
Similarly, if $\va w_{k_n} \vb = \va z_n \vb$ and   $M_{k_n} = \ca  \tl m : \va w_m \vb = \va w_{k_n} \vb \tr $, then
$$ \zeta_{D^+ (z_n ,r_n)} (f) = \frac{1}{\va z_n \vb^{M_{k_n}}} \frac{\va z_n \vb^{k_n}}{\prod_{i=1}^{k_n} \va w_i \vb} \ \xi_{D^+ (z_n ,r_n)} (f) . $$

Now, recall that, if $(a_n)$ is a decreasing sequence in $\er$ with $\sum_{n=1}^{\infty} a_n < + \infty$, then  $\lim_{n \ra \infty} n a_n =0$. Equivalently, since $\pl \va w_k \vb \pr$ is increasing and $\prod_{k=1}^{\infty} \va w_k \vb >0$,
$\lim_{n \ra \infty} \va w_{k_n} \vb^{k_n} =1$, so $\lim_{n \ra \infty} \va  z_n \vb^{k_n} =1$ and, consequently, $\lim_{n \ra \infty} \va z_n  \vb^{M_{k_n}} =1$.
On the other hand, since $\vc f \vd = 1 / \prod_{k=1}^{\infty} \va w_k \vb$, we easily conclude the result.
 \end{proof}

We give a final lemma that will be used later.

 \begin{lem}\label{correicion}
Let $z \in \de$, $z \neq 0$, and suppose that  $0 < s < r < \va z \vb$.
If $f \in \hac$, then
$$
  \pl \frac{s}{r}\pr^{\cer \pl f , D^- \pl z, r \pr \pr }
 \xi_{D^+ \pl z , r \pr } (f) \le \xi_{D^+ \pl z , s \pr} (f) \le \xi_{D^+ \pl z  , r \pr } (f).
$$
\end{lem}

\begin{proof}
It is clear that $\xi_{D^+ \pl z , s \pr} (f) \le \xi_{D^+ \pl z  , r \pr } (f)$. On the other hand, if $w_1, \ldots, w_n$
are the zeros of $f$ in $D^- \pl z , r \pr \setminus D^+ \pl z , s \pr$, and $z_1, \ldots, z_m$ are the zeros of
$f$ in $C(0, \va z \vb ) \setminus D^- \pl z , r \pr$, then
\begin{eqnarray*}
\xi_{D^+ \pl z  , s \pr} (f)
 &=&  {s}^{\cer \pl f , D^+ \pl z, s \pr \pr }
\prod_{i=1}^{n} \va  z - w_i \vb
 \prod_{j=1}^m \va z - z_j \vb  \\
 &\ge&  \pl \frac{s}{r}\pr^{\cer \pl f , D^- \pl z, r \pr \pr } {r}^{\cer \pl f , D^- \pl z, r \pr \pr}  \prod_{j=1}^m \va z - z_j \vb \\
 &=&  \pl \frac{s}{r}\pr^{\cer \pl f , D^- \pl z, r \pr \pr } \xi_{D^+ \pl z  , r \pr } (f) ,
 \end{eqnarray*}
and we are done.
\end{proof}

\section{$\ma$ and $\mathscr{M}^*$}\label{mieko}

Proposition~\ref{josejrut} is given in \cite{E3}. For the sake of completeness, we provide a (different) proof.

\begin{prop}\label{josejrut}
Suppose that $\varphi \in \ma$  satisfies  $\varphi  = \psi \in \mathscr{M}^*$ on $\ka [z]$. Then $\varphi  = \mathbf{i} \pl \psi \pr$.
\end{prop}

\begin{proof}
To see that $\varphi  =\mathbf{i} \pl \psi \pr$, it is enough to prove the equality at any  $f \in \hac$
satisfying $f(0) =1$ and having infinitely many critical
radii $R_j$. Since $\psi \neq \vc \ \vd $, we can find $r \in (0,1)$ with $\psi \le  \zeta_{D^+ (0, r)}$, and
  we may assume that
$f$ has  $m_j$ zeros in each $C(0, R_j)$, and  that $r < R_1 < R_2 < \cdots $. For each $R \in (R_1, 1)$, we write $f = P_R f_R$, where $P_R \in \ka [z]$ is the product
 $P_R (z) := \prod_{i=1}^n \pl z - z_i \pr$, being the $z_i$ all the zeros of $f$ in $D^+ (0, R)$,
 and $f_R \in \hac$  has no zeros in $D^+ (0, R) $.

\medskip

{\bf Claim.} {\em The limit $\ql \varphi  \qr (f) := \lim_{R \ra 1} \varphi   (f_R) $ exists, and
$$\varphi  (f) = \frac{\mathbf{i} \pl \psi \pr (f)}{\vc f \vd } \ql \varphi \qr (f).$$}

 \medskip

For $R \in (R_1, 1)$ fixed, let $N$ be the largest integer with $R_N \le R$, so that $R_1, \ldots, R_N$ are the critical radii of $P_R$.
Obviously $\va P_R \vb$ and $\va f \vb$ are constant in  $D^+ (0, r)$, so
 $  \psi \pl P_R \pr =  \va P_R (0) \vb =     \prod_{j=1}^{N} {R_j}^{m_j} $, and
  $\mathbf{i} \pl \psi \pr (f)  = \va f(0) \vb =1$.  Since $\vc f \vd  = 1/ \prod_{n=1}^{\infty} {R_n}^{m_n}$,    $\lim_{R \ra 1}  \psi  (P_R) \vc f \vd  =1 =\mathbf{i} \pl \psi \pr (f)$. 
Also $\varphi(f) = \psi (P_R) \varphi (f_R) $ for all $R$, so  by taking limits we prove the claim.
$\blacksquare$

\smallskip

Also, since $\vc P_R \vd =1$,   $\vc f_R \vd  = \vc f \vd $ for all $R$, and consequently $\ql \varphi \qr (f) \le \vc f \vd $ and   $\varphi (f) \le \mathbf{i} \pl \psi \pr (f)$. We easily conclude that $\varphi (g) \le \mathbf{i} \pl \psi \pr (g)$ whenever $g \in \hac$  has constant absolute value on $D^+(0,r)$.

Suppose next that $\varphi (f) < \mathbf{i} \pl \psi \pr (f)$, that is,
 $\varphi (f) < 1$. Note that $f (z) := 1 + \sum_{n=1}^{\infty} a_n z^n$ and, since there are no critical radii $R\le r$,
 $M:= \sup_{n \in \en} a_n r^n < 1$, so the function $h (z) := f(z) -1$
 satisfies $\va h(z) \vb \le M$ for all $z \in D^+ (0, r)$ and  $\mathbf{i} \pl \psi \pr (h) \le \zeta_{D^+ (0, r)} (h)  <1$.
We can write $h = P g$, where $P \in \ka [z]$ and $g \in \hac$ has constant absolute value
 in   $D^+(0,r)$, which implies that $\varphi (g) \le \mathbf{i} \pl \psi \pr (g)$. Obviously,  $ \psi (P) \mathbf{i} \pl \psi \pr (g)  = \mathbf{i} \pl \psi \pr (h) <1$, whereas  $ \psi (P) \varphi (g)  = \varphi (h) =1$ (because $\varphi (f) <1$ and $\varphi (1) =1$), implying that
 $\mathbf{i} \pl \psi \pr (g) < \varphi (g)$. Since this is impossible, we conclude that $\varphi (f) = \mathbf{i} \pl \psi \pr (f)$.
             \end{proof}

\begin{proof}[Proof of Theorem~\ref{casitadechocolate}]
It is obvious that $\mathbf{i}$ is injective and that $\mathbf{i}^{-1} :  \mathbf{i} \pl \mathscr{M}^* \pr \ra \mathscr{M}^* $ is continuous. Next, suppose that
 $\pl \zeta_{\lambda } \pr_{\lambda \in \Lambda}$ is a net in $\mathscr{M}^*$ convergent to $\zeta_{\lambda_0} \in \mathscr{M}^*$. By the definition of convergence of a net, since $\zeta_{\lambda_0} \neq \vc \ \vd$, there exist  $r \in (0,1)$ and $\lambda_1 \in \Lambda$  such that  $\zeta_{\lambda} \le \zeta_{D^+ (0, r)}$ for all $\lambda \ge \lambda_1$, and $\zeta_{\lambda_0} \le \zeta_{D^+ \pl 0, r \pr}$. This implies in particular that, for $g \in \hac$, if $\va g \vb$ is constant in $D^+ (0,r)$, then $\mathbf{i} \pl \zeta_{\lambda_0} \pr (g) = \va g (0) \vb = \mathbf{i} \pl \zeta_{\lambda} \pr (g)$ for all $\lambda \ge \lambda_1$.

Now consider  $f \in \hac$. Obviously $f = P g$ where $P$ is a polynomial with all its zeros in $D^+ (0, r)$ and $g \in \hac$ has no zeros in $D^+ (0,r)$. Then, taking into account that $P \in \ka [z]$, for $\lambda \ge \lambda_1$ and $\lambda = \lambda_0$,
$\mathbf{i} \pl \zeta_{\lambda} \pr (f) =  \zeta_{\lambda}   (P) \va g (0) \vb$. Consequently  $\pl \mathbf{i} \pl \zeta_{\lambda } \pr (f) \pr_{\lambda \in \Lambda}$ converges to $\mathbf{i} \pl \zeta_{\lambda_0 } \pr (f)$. The fact that $\mathbf{i}$ is continuous follows easily.

We next see that $\mathbf{i} \pl \mathscr{M}^* \pr$ is open in $\ma$.
Given $\varphi \in \mathscr{M}^*$, there exists  $r<1$ such that $\varphi \le \zeta_{D^+ \pl 0, r \pr}$ and a polynomial $P \in \ka [z]$
       with all   its zeros in $D^+ \pl 0, r \pr$ such that $\zeta_{D^+ \pl 0, r \pr} (P) < \vc P \vd /2$. Now if $\psi \in \ma$
       satisfies $\va \psi (P) - \mathbf{i} (\varphi) (P) \vb < \vc P \vd /2$, then $\psi (P) < \vc P \vd $, so the restriction of $\psi$ to $\ka [z]$ is not equal to $\vc \ \vd $. By Proposition~\ref{josejrut}, $\psi$ belongs to $\mathbf{i} \pl \mathscr{M}^* \pr$.

       Finally, the map $T: \ma \ra \mathscr{M} (T_1)$ that coincides with $\mathbf{i}^{-1}$ on $\mathbf{i} \pl \mathscr{M}^* \pr$ and sends $\ma \setminus  \mathbf{i} \pl \mathscr{M}^* \pr$ to $\vc \ \vd$ is easily seen to be continuous and closed. The result now follows from \cite[Proposition  2.4.3]{E}.
\end{proof}

\section{Sequences of zeros}\label{gatopo}

It is well known that, in complex analysis,  under some natural conditions, a bounded analytic function can be constructed
  to have zeros precisely at a given sequence $(z_n)$ of complex numbers in the open unit disk, each with a prescribed multiplicity
   (see \cite[Theorem II.2.2]{G}). A similar result does not hold for nonarchimedean fields, in particular when they are not spherically complete, as it is the case
   of the $p$-adic complex fields $\ce_p$ (see \cite{L}).
    Nevertheless, in the nonarchimedean context, an  analytic function (not necessarily bounded) can be found having as zeros the points of the   sequence $(z_n)$ when it satisfies a natural condition, but with multiplicities larger (and not necessarily equal) than those prescribed (see \cite{FM}, and \cite[Theorem 25.5]{E1} for a detailed proof).

   Roughly speaking, here  we are interested in finding $f \in \hac $ having zeros not at points of a given sequence $(z_n)$, but {\em close} to them, and paying
    attention instead to the   the fact that  any of those zeros is simple and that $\vc f \vd$ belongs to $\va \ka^{\times} \vb$.

We begin with a  well known  result  (see for instance \cite[p. 15]{Sh}).

\begin{lem}\label{tiadelphi}
Let $\gamma_1, \ldots, \gamma_n \in \ka$ be pairwise different. Then the rank of the Vandermonde matrix
\[ \left(
\begin{array}{lllll}
1 & \gamma_1 &  {\gamma_1}^2 & \ldots &{\gamma_1}^{n-1} \\
1 & \gamma_2 &  {\gamma_2}^2 & \ldots &{\gamma_2}^{n-1} \\
1 & \gamma_3 &  {\gamma_3}^2 & \ldots &{\gamma_3}^{n-1}\\
\vdots & \vdots &\vdots &\ddots &\vdots \\
1 & \gamma_n &  {\gamma_n}^2 & \ldots &{\gamma_n}^{n-1}
\end{array}
\right) \]
is $n$.
\end{lem}

Next, and throughout this section, we use the notation and basic properties of critical radii and  zeros of analytic functions given at the beginning of Section~\ref{complauton}.

\begin{lem}\label{branu}
Let $P (z) := a_0 + a_1 z+ \cdots +  z^n \in \ka [z]$. If $P(z) = \prod_{i=1}^n (z - z_i)$ with $z_1, \ldots, z_n \in \de$, then
$\va a_i \vb \le 1$ for all $i \in \tl 0, \ldots, n -1 \tr$.
\end{lem}

\begin{lem}\label{korralasign}
Let $P_1 (z) , Q(z)  \in \ka [z]$, where the degree of $P_1 (z)$ is $n>0$, and let
$$ P_2 (z) := P_1 (z) + z^{n +1} Q (z) .$$
Suppose that
$R_1$ is a critical radius of $P_2(z)$ satisfying $\mu_{R_1} (P_2) > n$ and that $C (0, R_1)$  contains exactly $k$ zeros of $P_2 (z)$, $k>0$. Then
it also contains exactly $k$ zeros of $Q(z)$.
\end{lem}

\begin{proof}
We write
$P_1(z) := a_0+ a_1 z + \ldots + a_n z^n $ and $ Q (z) := a_{n+1} + a_{n+2} z + \ldots + a_{n+m+1} z^m $, so that $P_2(z) = \sum_{i=0}^{n+m+1} a_i z^i$.
By definition,
$\va a_i \vb {R_1}^i < M_{R_1} (P_2)  $
for all   $i \notin \tl \mu_{R_1} (P_2), \ldots, \nu_{R_1} (P_2) \tr $.
Also
 $$M_{R_1} (P_2) = \va a_{\mu_{R_1} (P_2) } \vb {R_1}^{\mu_{R_1} (P_2)} =  \va a_{\nu_{R_1} ( P_2)  } \vb {R_1}^{\nu_{R_1} ( P_2)},$$
 and $\va a_i \vb {R_1}^{i}  \le  M_{R_1} (P_2)  $ for $i \in \tl \mu_{R_1} (P_2), \ldots, \nu_{R_1} (P_2) \tr $.
Taking into account that $\mu_{R_1} (P_2)  \ge   n+1$,  we easily see that $ \mu_{R_1} (Q) = \mu_{R_1} ( P_2) - n - 1$ and $\nu_{R_1} (Q) = \nu_{R_1} (P_2) - n - 1$. Since,
$\nu_{R_1} (P_2) = k + \mu_{R_1} (P_2)$, the conclusion follows easily.
\end{proof}

\begin{lem}\label{markapasos}
Let $M_1, M_2, M_3 \in \en$. Let
$P_1 (z) = p_1 (z) q_1 (z) = 1 + \sum_{n=1}^{M_1 + M_2} a_n z^n$,
where $p_1 (z) , q_1 (z) \in \ka [z]$
have
degrees $M_1$ and $M_2 $, respectively.
Let $A_1$ and $A_2$ be the sets of zeros of $p_1 (z)$ and $q_1 (z)$,
respectively, and suppose that each zero of $q_1 (z)$ is simple and
$\max_{z \in A_1} \va z \vb < \min_{z \in A_2}  \va z \vb$.

Suppose that $S \in \va \ka^{\times} \vb$ and that $A_3 \subset \ka$ has $M_3$ points and satisfy
 $\max_{z \in A_2} \va z \vb < S < \min_{z \in A_3}  \va z \vb$.

Then there exists $Q (z) \in \ka [z]$ of degree $M_2 + M_3 $ such that
$P_2 (z) := P_1 (z) + z^{M_1 + M_2 +1} Q (z)$ can be written as
$$P_2 (z) = p_2 (z)  q_2 (z) ,$$
with $p_2 (z) = r_2 (z) s_2 (z)$, where $M_1$, $M_2 +1$ and $M_2 + M_3$ are the degrees   of $r_2 (z)$, $s_2 (z)$ and $q_2 (z)$, respectively, and

\begin{itemize}
\item each $z \in A_2 \cup A_3$ is a simple zero of $q_2 (z)$;
\item $r_2 (z)$ has the same critical radii as $p_1 (z)$, and the same number of zeros in each critical radius;
 \item  all $M_2 +1$ zeros of $s_2 (z)$ are contained in $C(0, S)$.
\end{itemize}
\end{lem}

\begin{proof}
We suppose that
\begin{eqnarray*}
 \tl \va z \vb : z \in A_1 \tr  &=& \tl  R_1, \ldots, R_{N_1} \tr  \\
  \tl  \va z \vb : z \in A_2 \tr  &=& \tl R_{N_1 +1 }, \ldots, R_{N_2} \tr  \\
    \tl \va z \vb : z \in A_3 \tr &=& \tl R_{N_2 +1 }, \ldots, R_{N_3} \tr ,
   \end{eqnarray*}
 with $R_1 < \cdots < R_{N_3}$,  and that for each $j \in \tl N_1 +1 , \ldots , N_3 \tr$, there are $k_j$ (pairwise different) points $z \in A_2 \cup A_3$ with $\va z \vb = R_j$. Also,
   for each $j \in \tl 1, \ldots, N_1 \tr$, there are $k_j$ zeros in $ A_1 $ with absolute value $ R_j$.

Fix $w_1 \in \ka$ with $\va w_1 \vb =S$, so $R_{N_2} < \va w_1 \vb < R_{N_2 +1}$.
According to Lemma~\ref{tiadelphi}, there exist $M_2  +  M_3 + 1$ coefficients  $b_0,
\ldots, b_{M_2 + M_3} \in \ka$ such that $b_0 + b_1 z + \cdots + b_{M_2 + M_3} z^{M_2 + M_3} = -  P_1 (z)/ z^{M_1 + M_2 +1}  $
for all $z \in A_2 \cup A_3 \cup \tl w_1 \tr$, that is,
$$ P_2 (z) := P_1 (z) + z^{M_1 + M_2 +1} \pl \sum_{n=0}^{M_2 + M_3} b_n z^{n} \pr =0 .$$

Since   $R_{N_2 +1}$ is bigger than $\va w_1 \vb$ and  $\mu_{\va w_1 \vb} (P_1 ) = \nu_{\va w_1 \vb} (P_1 ) = M_1 + M_2 $,
$$\mu_{R_{N_2 + 1}} (P_2) > \mu_{\va w_1 \vb} (P_2) \ge \mu_{\va w_1 \vb} (P_1 )   = M_1 + M_2 ,$$
and
we conclude from Lemma~\ref{korralasign} that  $Q_0 (z) := \sum_{n=0}^{M_2 + M_3} b_n z^n $  has exactly $k_i$ zeros in each $C(0, R_i)$ for $i=N_2 + 1, \ldots, N_3$.
On the other hand,  each $z \in A_2$ is a zero of $Q_0 (z)$, and the degree of $Q_0 (z)$ is $M_2 + M_3$, so  $Q_0 (z)$ has exactly $k_i$ zeros in each $C(0, R_i)$ for $i=N_1 + 1, \ldots, N_3$. Since it has no other zeros, again by Lemma~\ref{korralasign},
 if $S_1 > R_{N_2}$ with $S_1 \neq R_{N_2 +1}, \ldots, R_{N_3}$ is a critical radius for $P_2$, then $\mu_{S_1} (P_2) \le M_1 + M_2$, implying that $\nu_{R_{N_2}} (P_2) \le M_1 + M_2$. On the other hand, since  $\nu_{R_{N_2}} (P_2) \ge \nu_{R_{N_2}} (P_1)$, we deduce that $\nu_{R_{N_2}} (P_2) =  M_1 + M_2 = \mu_{S_1} (P_2)$. We conclude that $\va w_1 \vb$ is the only critical radius for $P_2$ bigger than $R_{N_2}$ and different from all other $R_i$, which necessarily gives
$\mu_{\va w_1 \vb} (P_2) = M_1 + M_2$ and $\nu_{\va w_1 \vb} (P_2) = \mu_{R_{N_2 +1}} (P_2) = M_1 + 2 M_2 +1$. This
 implies that
$P_2 (z)$ has $M_2 + 1$ zeros in $C(0, \va w_1 \vb)$.

 Finally, since $\nu_{R_{N_2}} (P_2)   = M_1 + M_2 $ we have that $\va a_{M_1 + M_2} \vb {R_{N_2}}^{M_1 + M_2}  > \va b_n \vb {R_{N_2}}^{M_1 + M_2 + n +1}$ for all $n \ge 0$, which implies that $$\va a_{M_1 + M_2} \vb R^{M_1 + M_2}  > \va b_n \vb R^{M_1 + M_2 + n +1}$$
 whenever $0 < R< R_{N_2}$. Consequently, critical radii of $P_2 (z)$ and $P_1 (z)$ in $\pl 0, R_{N_2} \qr$ coincide, as well as the number
 of zeros in each critical radius. This means that
 each $C(0, R_i)$ contains exactly $k_i$ zeros of $P_2 (z)$, for $i=1 , \ldots, N_2$.
\end{proof}

\begin{prop}\label{gatillo}
Let $ \ze $ be a sequence in $\de$ with $c:= \prod_{n=1}^{\infty} \va z_n \vb >0$. Suppose that
the   disks $D^+ \pl z_n , \epsilon_n \pr$, $n \in \en$, are pairwise disjoint.
Then there exists $f \in \hac$ with $ f (0) =1$ and $\vc f \vd \in \va \ka^{\times} \vb$ having exactly a single zero in each $D^+ \pl z_n, \epsilon_n \pr$ and such that, for any other zero $z$ of $f$, $\va z \vb \neq \va z_n \vb $ for every $n \in \en$.
\end{prop}

\begin{proof}
 Let $\tl R_i : i \in \en \tr = \tl \va z_n \vb : n \in \en \tr $, and suppose that, for each $i$,  $R_i < R_{i+1}$
 and
$z_1^i, \ldots , z_{k_i}^i  $ are those $z_n$ of absolute value $R_i$. We  select $\delta_i \in \va \ka^{\times}
 \vb$, $\delta_i \le \min \tl \epsilon_n : \va z_n \vb = R_i \tr $, and assume also that $\delta_1 < R_1$.

Pick any $N_1 \in \en$ and define $M_1 := \sum_{i=1}^{N_1} k_i$. Then take
$N_2 > N_1$ such that, for $M_2 := \sum_{i=N_1 +1}^{N_2} k_i$, ${R_{N_1}}^{M_2} < c  \, \min_{1 \le i \le N_1}   {\delta_i}^{k_i} $.

  Inductively, for any other  $n \in \en$, pick $N_{n+1} > N_n$ such that
 $${R_{N_n}}^{M_{n+1}} \le  c \min_{N_{n -1} +1 \le i \le N_n }    {\delta_i}^{k_i} ,$$
where $M_{n+1} := \sum_{i= N_n +1}^{N_{n+1}} k_i$.

Based on the sequence $\pl R_{N_n} \pr$, we fix a new sequence $(S_n)_{n \ge 2}$ in $\va \ka^{\times} \vb$
with
$$  R_{N_n} < S_n < R_{N_n +1 }$$
for all $n \ge 2$.
Next  call $N_0 :=0$ and, for $n \ge 1$,   $$B_n := \tl R_i : N_{n-1} +1 \le i \le N_n \tr$$ and
$$A_n := \tl z_1^i, \ldots , z_{k_i}^i : N_{n -1} +1 \le i \le N_n \tr .$$

Clearly, the polynomial
$$P_1 (z) := \prod_{i=1}^{N_2} \pl \prod_{j=1}^{k_i}   \pl 1 - \frac{z}{z^i_j}\pr \pr$$
has degree $M_1 + M_2$ and its (simple) zeros are the points in $A_1 \cup A_2$.

We write $P_1 (z) := 1 + a_1 z + \cdots + a_{M_1 +M_2} z^{M_1 + M_2}  $.
Next, by Lemma~\ref{markapasos}, we can inductively construct a sequence $(P_n)$ in $\ka [z]$ such that, for all $n \ge 2$ and $L_n := M_1 + 2 M_2 + \cdots + 2 M_n + M_{n+1} + n-1$,
$$P_n (z)  : = 1 + a_1 z + \cdots + a_{ L_n} z^{L_n} $$ can be written as
\begin{equation*}
P_n (z)   = p_{n} (z)    q_{n} (z) = r_{n} (z) s_n (z)  q_{n} (z),
\end{equation*}
for some polynomials
$$ q_n (z) := \prod_{i \in B_n \cup B_{n+1}} \pl \prod_{j=1}^{k_i} \pl 1 - \frac{z}{z_j^i} \pr \pr$$
of degree $M_n + M_{n+1}$ having all points in $A_n \cup A_{n+1}$ as simple zeros,
$$r_{n} (z):= \prod_{i \in B_1 \cup \cdots \cup B_{n-1}} \pl \prod_{j=1}^{k_i} \pl 1 - \frac{z}{w_j^i} \pr \pr$$
 of degree $M_1 + M_2 +  \cdots   + M_{n-1}$ and with critical radii
$R_i$ for all $i \le N_{i-1}$, having $k_i$ (not necessarily simple) zeros $w_j^i$  in each $C(0, R_i)$, and
$$s_{n} (z):= \prod_{i = 2}^{n} \pl \prod_{j=1}^{M_i +1} \pl 1 - \frac{z}{u_j^i} \pr \pr$$
 of degree $M_2 + M_3 +  \cdots   + M_{n} + n -1$ and with critical radii
$S_i$ ($i \le n$), having $M_i +1$ (not necessarily simple) zeros $u_j^i$  in each $C(0, S_i)$.

Hence, for $l \in B_n$ and $\va z \vb = R_l$,
\begin{eqnarray*}
\va q_n (z) \vb &=& \prod_{i = N_{n-1} +1}^{l} \pl \prod_{j=1}^{k_i} \va 1 - \frac{z}{z_j^i} \vb \pr \\
&=& \frac{{R_l}^{k_{N_{n -1} +1 } + \cdots +k_{l -1} }}{{R_{N_{n-1} +1 }}^{k_{N_{n -1}+ 1}} \cdots {R_{l -1}}^{k_{l-1}}}  \prod_{j=1}^{k_l} \va \frac{z - z_j^l }{z_j^l} \vb ,
\end{eqnarray*}
 $$\va r_{n} (z) \vb = \frac{{R_l}^{M_1 + \cdots +M_{n-1}}}{{R_1}^{k_1} \cdots {R_{N_{n-1}}}^{k_{N_{n -1}}}} $$ and
$$\va s_n (z) \vb= \prod_{i = 2}^{n-1} \pl \prod_{j=1}^{M_i +1} \va 1 - \frac{z}{u_j^i} \vb \pr =
\frac{{R_l}^{M_2 + \cdots +M_{n-1} + n -2}}{{S_2}^{M_2 +1} \cdots {S_{n-1}}^{M_{n-1} +1}} .
$$

This implies that
\begin{equation}\label{topivacan}
\va P_n (z) \vb =  \va s_n (z) \vb  T_l   \prod_{j=1}^{k_l} \va z - z_j^l  \vb  ,
\end{equation}
where $$T_l := \frac{{R_l}^{k_1 + \cdots + k_{l -1}}}{{R_1}^{k_1} \cdots {R_{l }}^{k_{l}}}.$$

On the other hand, if   in addition
$\va z - z_j^l \vb \ge  \delta_l $ for all $j$, then   $\prod_{j=1}^{k_l} \va z - z_j^l \vb \ge {\delta_l}^{k_l} $.
Taking into account that
$$\va a_{L_n} \vb =
\frac{1}{{S_2}^{M_2 +1} \ldots {S_n}^{M_n +1} {R_1}^{k_1} \cdots {R_{N_n+1}}^{k_{N_{n+1}}}} ,$$
 we easily obtain
\begin{eqnarray*}
\va a_{L_n} \vb {R_l}^{L_n} \frac{1}{\va s_n (z) \vb}\frac{1}{T_l}
&=& \frac{{R_{l}}^{M_1 + M_2  +\cdots +  M_{n-1} + 2M_n+  M_{n+1} +1} }{ {S_n}^{M_n +1} {R_{l +1}}^{k_{l+1}} \cdots {R_{N_{n+1}}}^{k_{N_{n+1}}}    {R_l}^{k_1 + \cdots + k_{l -1}}}   \\
&<& \frac{{R_{l}}^{k_l + \cdots + k_{N_{n+1}}} }{{R_{l +1}}^{k_{l+1}} \cdots {R_{N_{n+1}}}^{k_{N_{n+1}}} }
\\ &<&    \frac{{R_{N_n}}^{M_{n+1}}}{c}   \\
&\le & {\delta_l}^{k_l} \\
&\le&    \prod_{j=1}^{k_l} \va z - z_j^l  \vb .
 \end{eqnarray*}

This implies,  by Equality~\ref{topivacan},
\begin{equation}\label{berurubio}
\va a_{L_n} \vb  {R_{l}}^{L_n} < \va P_n (z) \vb.
\end{equation}

Next, using  Lemma~\ref {branu} it is  easy to see that, since each $P_n (z)$ has all its zeros contained in $\de$ and $P_n (0) =1$,
all its coefficients satisfy $$\va a_i \vb \le  \frac{1}{\prod_{j=1}^{\infty} {R_j}^{k_j} \prod_{j=2}^{\infty} {S_j}^{M_j +1}} \le   \frac{1}{c^3} ,$$
which implies
that $f(z) := 1 + \sum_{m=1}^{\infty} a_m z^m$ is bounded and, consequently, belongs to $\hac$. Also, the critical radii for $f$ are the $R_i$ and the $S_i$, and
it has exactly $k_i$ zeros in each $C(0, R_i)$ and $M_i + 1$ zeros in each $C(0, S_i)$. We now define, for
$n \in \en$, $$g_n (z) := f (z) - P_n (z) = \sum_{m= L_n + 1 }^{\infty} a_m z^m .$$

Note that, since  $\va a_{L_n} \vb  {R_{N_{n+1}}}^{L_n} > \va a_m \vb {R_{N_{n+1}}}^m$ for all $m > L_n$,
$$\va a_{L_n} \vb  {R_{l}}^{L_n} > \va a_m \vb {R_{l}}^m $$
for  $l \in B_n$, and consequently,
 if   $\va z \vb = R_l$, then
 \begin{equation}\label{xinesa}
\va g_n (z) \vb <
\va a_{L_n} \vb  {R_{l}}^{L_n} .
\end{equation}
 We deduce from Inequalities~\ref{berurubio} and~\ref{xinesa} that
$\va a_{L_n} \vb  {R_{l}}^{L_n} < \va f (z) \vb$
whenever
 $\va z \vb = R_l $ and
$\va z - z_j^l \vb \ge \delta_l $ for all
$j \in \tl 1, \ldots , k_l \tr$.
In particular, if we fix $j \in  \tl 1, \ldots, k_l \tr$ and take
$w \in \de$ with $\va w - z_j^l \vb = \delta_l $,
then $\va f(w) \vb > \va a_{L_n} \vb  {R_{l}}^{L_n} > \va g_n (z_j^l) \vb$. On the other hand,
$P_n \pl z_j^l \pr =0$, so
$ \va f \pl z_j^l \pr \vb = \va g_n \pl z_j^l \pr \vb $.
This means that, if we define
$h (z) := f \pl z + z_j^l \pr$,
then $  \va h (0) \vb < \va h(z) \vb$
whenever $\va z \vb = \delta_l $.
We conclude that either $h(0) =0$ or there is a critical radius for $h$ between $0$ and
$\delta_l $, and consequently
there is a zero of $f$ in
 $D^- \pl z_j^l , \delta_l \pr $.
 Since $f$ has exactly $k_l$ zeros in $C(0, R_l)$, we are done.

 It just remains to prove that the above $f$ can be taken to satisfy $\vc f \vd \in \va \ka^{\times} \vb$.
 Note that
 apart from
the $R_n$, the critical radii of the function $f$ are
 certain $S_n \in \va \ka^{\times} \vb \cap \pl R_{N_n} , R_{N_n +1} \pr$, $n \ge 2$, chosen  at will. Let us next see  that these can be selected in such a way that
$\vc f \vd = 1/ \pl \prod_{n=1}^{\infty} {R_n}^{k_n} \prod_{n=2}^{\infty} {S_n}^{M_n +1} \pr $ belongs to $\va \ka^{\times} \vb$. Clearly, it is enough to show that every value in the interval $\pl \prod_{n=2}^{\infty} {R_{N_n}}^{M_n +1} , \prod_{n=2}^{\infty} {R_{N_{n+1}}}^{M_n +1} \pr$ is attained by products of the form $\prod_{n=2}^{\infty} {S_n}^{M_n +1}$. It is easy to see that this is equivalent to proving that, given a set $D$ dense in $(0,+ \infty)$, if $(a_n)$ and $(b_n)$ are sequences in $(0, + \infty)$ with $\sum_{n=1}^{\infty} b_n < \infty$ and $0< a_n < b_n$ for all $n$, then every $T \in \pl \sum_{n=1}^{\infty} a_n, \sum_{n=1}^{\infty} b_n \pr$ can be written in the form $T = \sum_{n=1}^{\infty} q_n (t_n)$ with all $q (t_n) \in D$, where $q_n (s):= s a_n + (1-s) b_n$ for every $s \in [0,1]$ and $n \in \en$.

First, it is clear that there exists $s_1 \in (0,1)$ with $T = \sum_{n=1}^{\infty} q_n (s_1)$.
We fix $\epsilon >0$
 and pick $t_1 \in [0,1] $ with  $q_1 (t_1) \in D $ and $q_1 (t_1) + q_2 (0) < q_1 (s_1) + q_2 (s_1) < q_1 (t_1) + q_2 (1)$ such that $\va q_1 (t_1) - q_1 (s_1) \vb <   \epsilon$. Then there exists $s_2 \in (0,1)$ with $q_1 (t_1) + q_2 (s_2) = q_1 (s_1) + q_2 (s_1)$,  that is,
 $ q_1 (t_1) +    q_2 (s_2) +   q_3 (0) < \sum_{n=1}^{3} q_n (s_1)   < q_1 (t_1) +   q_2 (s_2)  +   q_3  (1)$. Consequently,
 there exists $t_2 \in (0,1)$ with $q_2 (t_2) \in D $ such that
$$ q_1 (t_1) +    q_2 (t_2) +   q_3 (0) < \sum_{n=1}^{3} q_n (s_1)   < q_1 (t_1) +   q_2 (t_2)  +   q_3  (1)$$
and
  $\va  q_1 (t_1) +    q_2 (t_2) - \sum_{n=1}^{2} q_n (s_1) \vb < \epsilon/2 $.

Clearly, we inductively find a sequence $(t_n)$ in $(0,1)$ with $q_n (t_n) \in D$ for each $n$, and such that
$
 \va  \sum_{n=1}^k q_n (t_n)  - \sum_{n=1}^k q_n (s_1) \vb <
\epsilon/k
$
for all $k$.
Thus $T= \sum_{n=1}^{\infty} q_n (t_n)$, and we are done.
\end{proof}

\begin{rem}\label{kuku}
Note that, for a sequence $(T_n)$ in $(0,1)$, the function $f $  in Proposition~\ref{gatillo} can be taken so  that no $T_n$ is a critical radius for $f$.
\end{rem}

\section{Sequences determining the same seminorms}\label{nomarkator}

In this section we first show that a seminorm is determined by the behaviour of the radii of seminorms along an ultrafilter.

 \begin{lem}\label{pijamalari}
Let $z_0 \in \de \setminus \tl 0 \tr$ and $s, r \in (0,1)$ satisfy $s \le  r < \va z_0 \vb$.
Then,  for every $f \in \hac \setminus \tl 0 \tr$,
$$
0 \le  \zeta_{D^+ (z_0, r)} (f) - \zeta_{D^+ (z_0, s)} (f)   \le
 \pl {r}^{\cer \pl f, D^+ \pl z_0, r \pr \pr } - {s}^{\cer \pl f, D^+ \pl z_0, r \pr \pr} \pr \vc f \vd .
$$
 \end{lem}

\begin{proof}
We
write $f(z)  = g (z) \prod_{i=1}^{m} (z- w_i)$, where $w_1,  \ldots , w_m$ are the zeros of $f$ in $C(0, \va z_0 \vb)$.
Taking into account Corollary~\ref{angiosperma}, it is easy to see that
$ \zeta_{D^+ \pl z_0, r \pr} (g)   =  \zeta_{D^+ \pl 0, s \pr} (g) $. Consequently,
$$
\zeta_{D^+ (z_0, r)} (f)
= \zeta_{D^+ (z_0, r)} (g) \ r^{\cer \pl f, D^+ \pl z_0,  r \pr \pr}   \prod_{w_i \notin D^+ ( z_0 , r)} \va z_0 - w_i \vb $$
and that
$$
\zeta_{D^+ (z_0, s)} (f)
\ge   \zeta_{D^+ (z_0, r)} (g)  \ s^{\cer \pl f, D^+ \pl z_0,  r \pr \pr} \prod_{w_i \notin D^+ ( z_0 , r)} \va z_0 - w_i \vb .$$
Since $\vc f \vd = \vc g \vd$, the conclusion follows easily.
\end{proof}

  \begin{cor}\label{pernia}
Let  $\kai$ be a sequence in $\en$. Suppose that $\ef $ is a nonprincipal  ultrafilter in $\en$ and that
$(r_n)$ and $
 (s_n)$ are sequences in $(0,1)$ such that $\lim_{\ef} {r_n}^{k_n} = \lim_{\ef} {s_n}^{k_n} \neq 0, 1$.
If $\ze$ is a sequence in $\de$ with
$r_n , s_n < \va z_n \vb$ for all $n$,
then $\lim_{\ef} \zeta_{D^+ (z_n , r_n )}   = \lim_{\ef} \zeta_{D^+ (z_n , s_n )}  $.
 \end{cor}

 \begin{proof}
We can assume that $A \in \ef$ satisfies
$s_n \le r_n  $
 for every $n \in A$.

 Let $f \in \hac$, $f \neq 0$.
By Lemma~\ref{pijamalari}, for $n \in A$,
$$
0 \le  \zeta_{D^+ (z_n, r_n)} (f) - \zeta_{D^+ (z_n, s_n)} (f)   \le
 \pl {r_n}^{\cer \pl f ,  D^+ (z_n, r_n) \pr} - {s_n}^{\cer \pl f ,   D^+ (z_n, r_n) \pr} \pr \vc f \vd .
$$
Obviously, from the hypothesis we deduce that $\lim_{\ef}  {r_n}^{t_n} = \lim_{\ef}  {s_n}^{t_n}$ for every sequence $(t_n)$ of natural numbers, and the conclusion follows.
\end{proof}

\begin{rem}\label{ugenio}
In the case when $\ka$ is not spherically complete, a natural question is whether the limit of norms based on filters in $\de$ with no center
allows us to define new seminorms. We will see that this is not the case.
Suppose that, for each $n \in \en$,  $\vc \cdot \vd_n = \lim_{m \ra \infty} \zeta_{D^+ \pl z_m^n , s_m^n \pr}$, where
$$C \pl 0, \va z_1^n \vb \pr \supset D^+ \pl z_1^n , s_1^n \pr \supset D^+ \pl z_2^n , s_2^n \pr \supset \cdots $$
and $\bigcap_{m=1}^{\infty} D^+ \pl z_m^n , s_m^n \pr = \emptyset$.
Suppose also that $\lim_{n \ra \infty} \va z_1^n \vb =1$. Take a nonprincipal  ultrafilter $\ef$ in $\en$ and define the seminorm $\psi:= \lim_{\ef} \vc \cdot \vd_n \in \ma$.
It is clear that $s_n := \lim_{m \ra \infty} s_m^n >0$ for each $n$. Consider a sequence $(k_n)$ in $\en$ such that
$r:= \lim_{\ef} {s_n}^{k_n} \in (0,1)$ and take, for each $n$, an $m_n \in \en$ such that
$ \lim_{\ef} {s_{m_n}^n}^{k_n} = r$. Calling $r_n := s_{m_n}^n$ and $z_n := z_{m_n}^n$, we easily check that
$\psi = \lim_{\ef} \zeta_{D^+ \pl z_n , r_n \pr}$.
\end{rem}

\begin{cor}\label{byemarcal}
Given $\varphi = \lim_{\ef} \zeta_{D^+ \pl z_n , r_n \pr} \in \mil$, there exist a sequence $(w_n) $ in $\de$ with $\lim_{n \ra \infty } \va w_n \vb =1$ and a sequence $(s_n)$ in $(0,1)$ in such a way that all the disks $D^+ \pl w_n , s_n \pr$ are pairwise disjoint and $\varphi = \lim_{\ef} \zeta_{D^+ \pl w_n , s_n \pr}$.
\end{cor}

\begin{proof}
Note that, if $\varphi = \vc \cdot \vd$, then $\varphi = \lim_{\ef} \zeta_{D^+ \pl z_n , \va z_n \vb^2 \pr} $, so
in all cases we can  assume without loss of generality that $D^+ \pl z_n , r_n \pr \subset C \pl 0 , \va z_n \vb \pr$ for all $n$. Fix $n_0 \in \en$, and suppose that the set
$$\tl n \in \en : \va z_n \vb = \va z_{n_0} \vb \tr = \tl n_1, \ldots , n_k \tr .$$
It is straightforward to prove that, for each $i \in \tl 1, \ldots , k \tr$, there exists $w_{n_i} \in C \pl z_{n_i} , r_{n_i} \pr$
such that $\va w_{n_{i_1}} - w_{n_{i_2}} \vb = \max \tl r_{n_{i_1}} , r_{n_{i_2}} , \va z_{n_{i_1}} - z_{n_{i_2}} \vb \tr$ whenever $i_1 \neq i_2$. This implies that the disks $D^- \pl w_{n_i} , r_{n_i} \pr$ are pairwise disjoint. Of course we can define a sequence $(w_n)$ with the desired properties by putting $w_n := w_{n_i}$ when $z_n = z_{n_i}$. Obviously, $\varphi = \lim_{\ef} \zeta_{D^- \pl w_n , r_n \pr}$. Now, if we assume that $\lim_{\ef} r_n >0$, then the conclusion follows immediately  taking into account Corollary~\ref{pernia}.
The case when $\lim_{\ef} r_n =0$ is similar.
\end{proof}

\begin{rem}\label{nonada}
Note that, in Corollary~\ref{byemarcal}, if $\ze$ is regular with respect to $\ef$, then so is $\we$. Taking into account that each  $\varphi = \zeta_{\ze, \ef}^{\kai, r} \in \mak$ can be written by $\varphi = \lim_{\ef} \zeta_{D^+ \pl w_n , s_n \pr}$, where all the disks $D^+ \pl w_n , s_n \pr$ are pairwise disjoint, we conclude that $\varphi \in \mim$. Thus, $\mak \subset \mim$.
\end{rem}

\begin{cor}\label{puigdemont}
Every $\varphi \in \mil$ can be written by $\varphi = \lim_{\ef} \zeta_{D^+ \pl z_n, r_n \pr}$, where $\ef \in \beta \en \setminus \en$ and  the disks $D^+ \pl z_n, r_n \pr$ are pairwise disjoint.
\end{cor}

Next we show that the converse of Corollary~\ref{pernia} does not hold in general. In fact, very different behavior of the radii along an ultrafilter can lead to the same seminorm (see Example~\ref{noterrias}
and
 Remark~\ref{saradientes}; see also Theorem~\ref{juntafaculta}).

\begin{prop}\label{lucile}
Let $\ze$ be a  sequence in $\de$ with $\lim_{n \ra \infty} \va z_n \vb =1$, and let $  \kai$ be a sequence in $\en$. Suppose that $\ef$ is a nonprincipal ultrafilter in $\en$ with the property that, for every $C \in \ef$,
$$ \lim_{\ef} \prod_{\substack{\va z_m \vb = \va z_n \vb   \\ m \in C \\  m \neq n }} \va z_n - z_m \vb^{k_m} < 1.$$
Let $(r_n)$ and $(s_n)$   be sequences in $(0,1)$ with $z_m \notin D^- (z_n, r_n), D^- (z_n, s_n)$ whenever $m \neq n$. If there exists
$C_0:= \tl n_1, n_2, \ldots, n_i, \ldots \tr \in \ef$      such that
$$\lim_{i \ra \infty} {r_{n_i}}^{k_{n_i}} =1 = \lim_{i \ra \infty} {s_{n_i}}^{k_{n_i}} ,$$ then
$$\lim_{\ef} \zeta_{D^+ (z_n , r_n )} = \lim_{\ef} \zeta_{D^+ \pl z_n ,  s_n \pr} .$$
\end{prop}

\begin{rem}\label{xuxo}
Note that in Proposition~\ref{lucile}, if  $\ze  $ is regular and $\kai  $ belongs to $\com_{\ef} (\ze)$, then the seminorm $\phi := \lim_{\ef} \zeta_{D^+ (z_n , r_n )}$    satisfies $\zeta_{\ze, \ef}^{\kai, 1} \le \phi$. In Example~\ref{sandaliuca}, we will see that the equality does not hold in general.
\end{rem}

\begin{rem}\label{susichon}
In Example~\ref{sandaliuca}, we will also see that a weaker assumption in Proposition~\ref{lucile} such as that $\lim_{\ef} {r_n}^{k_n} =1 = \lim_{\ef} {s_n}^{k_n}$ does not imply that
$\lim_{\ef} \zeta_{D^+ (z_n , r_n )} = \lim_{\ef} \zeta_{D^+ \pl z_n ,  s_n \pr} $.
\end{rem}

\begin{proof}[Proof of Proposition~\ref{lucile}]
Since we are dealing with an ultrafilter, we can assume without loss of generality that  $0  < s_n \le r_n $ for all $n \in C_0$. It is clear that   there exists
 a sequence $(l_n) $ in $\en$ with $\lim_{n \ra \infty} l_n = + \infty$ such that $ \lim_{\substack{n  \in C_0 \\ n \ra \infty}} {s_n}^{k_n l_n}  = 1/2$.

Take $f \in \hac$. We have that, if $\cer_n = \cer \pl D^- \pl z_n, r_n \pr \pr$ and
$$\lambda_n   := \prod_{\substack{ \va z_m \vb = \va z_n \vb  \\ m \neq n}} \va z_n - z_m \vb^{\cer_m} $$ and
$$\mu_n   := \prod_{\substack{ z \in \cer \pl C \pl 0, \va z_n \vb \pr \pr \\ \va   z - z_m \vb \ge  r_m \forall m}} \va z_n - z \vb $$
for  $n \in C_0$, then
\begin{equation}\label{gervasiafdezgtez}
{s_n}^{\cer_n} \lambda_n \mu_n \le \xi_{D^+ (z_n , s_n)} (f)
 \le \xi_{D^+ (z_n , r_n)} (f) = {r_n}^{\cer_n} \lambda_n \mu_n .
 \end{equation}
Let $\alpha := \lim_{\ef}  \cer_n / (k_n l_n)  $.
We easily see that, if $\alpha =0$,
then $\lim_{\ef} {s_n}^{\cer_n} =1$ and, taking limits in Equation~\ref{gervasiafdezgtez}, $\lim_{\ef} \xi_{D^+ (z_n , s_n)} (f) = \lim_{\ef} \xi_{D^+ (z_n , r_n)} (f)$.

On the other hand,   if $ 0 < \alpha \le + \infty$, then there exist $A \in \ef$ with $A \subset C_0$ and $\beta >0$ such that
$\cer_n \ge   \beta  k_n l_n$ for all $n \in A$. Next, for  $n \in A$ we define
$$L_n := \min \tl l_m :  m \in A , \va z_m \vb = \va z_n \vb \tr ,$$ and obtain
  $$\lambda_n \le \prod_{\substack{ \va z_m \vb = \va z_n \vb   \\m \in A \\ m \neq n}} \va z_n - z_m \vb^{\beta k_m L_n } = \pl \prod_{\substack{ \va z_m \vb = \va z_n \vb   \\m \in A \\ m \neq n}} \va z_n - z_m \vb^{k_m} \pr^{\beta  L_n } .$$

 Since $\lim_{n \ra \infty} l_n = +\infty$, $\lim_{\substack{ n \in A \\ n \ra \infty }} L_n = + \infty$. Also, by hypothesis, there exist $M < 1$ and $A' \in \ef$ with $A' \subset A$ and
 $$ \prod_{\substack{ \va z_m \vb = \va z_n \vb   \\m \in A \\ m \neq n}} \va z_n - z_m \vb^{k_m} \le M$$ for all $n \in A'$.
 This gives $\lim_{\substack{   n \in A' \\ n \ra \infty}} \lambda_n =0$, and consequently $\lim_{\ef}   \lambda_n =0$.

  Finally, it follows from Equation~\ref{gervasiafdezgtez} that $$\lim_{\ef} \xi_{D^+ (z_n , r_n)} (f) =0 = \lim_{\ef} \xi_{D^+ (z_n , s_n)} (f) =0 ,$$ and we are done.
\end{proof}

We next give an example where Proposition~\ref{lucile} can be applied.

\begin{ex}\label{noterrias}
Let $\ze$ be a sequence in $\de$ with $\prod_{n=1}^{\infty} \va z_n \vb >0$. Let $R_1 < R_2 < \cdots $ be the absolute values
of the $z_n$ and, for each $i$, suppose that there are $M_i \ge 2$ points $z_n$ of absolute value $R_i$, and that $\lim_{i \ra \infty} M_i = + \infty$. Suppose also that there exists
 $M \in \va \ka^{\times} \vb \cap (0,1)$  such that, for all $i \in \en$,
$\va z_n - z_m \vb = \sqrt[M_i -1]{M} \in (0,R_i)$ whenever $\va z_n \vb = R_i = \va z_m \vb$, $n \neq m$.

Fix a nonprincipal ultrafilter $\af$ in $\en$ and  consider the family $\mathfrak{F}$
 of the complements of all sets $D$ in $\en$ with the property
that
$$\lim_{\af} \frac{\ca (\tl z_n : n \in D \tr  \cap C(0, R_i))}{M_i} =0.$$
It is easy to check that
$\mathfrak{F}$ is a filter in $\en$, and that any ultrafilter $\ef$ containing $\mathfrak{F}$ satisfies the conditions of Proposition~\ref{lucile} for $k_n =1$ for all $n$. Thus, if   $r_n := \sqrt[M_i -1]{M}$ for each $n$ with $\va z_n \vb = R_i$, then $\lim_{\ef} \zeta_{D^+ (z_n , r_n )} = \lim_{\ef} \zeta_{D^+ \pl z_n ,  s_n \pr} $ for any sequence $(s_n)$ such that $\lim_{\ef} s_n =1$ and $s_n \le r_n$ for all $n$.
\end{ex}

\begin{prop}\label{gertrudisrivota}
Let $\ze$ be a   sequence in $\de$ with
$\lim_{n \ra \infty} \va z_n \vb =1$,
and let $\ef$ be a nonprincipal ultrafilter in $\en$.
Suppose that $ \pl D^- \pl z_n, r_n \pr \pr$ is a sequence of pairwise disjoint open disks.

If $\ze$ is  not regular with respect to $ \ef$, then   $$\lim_{\ef} \zeta_{ D^+ \pl z_n, r_n \pr} =   \lim_{\ef} \delta_{z_n} .$$
\end{prop}

\begin{proof}
For $f \in \hac $ with $\va f(0) \vb =1 $ fixed, let $C $ be the set of all $n$  such that $D^- (z_n, r_n)$ contains no zeros of $f$. Suppose first that $C$  belongs to $\ef$. Then $\va f \vb$ takes a constant value on each disk
$D^- \pl z_n, r_n \pr$, for $n \in C$, and this same value is taken at each $z_n  $, $n \in C$.
It is now straightforward to see that $\lim_{\ef} \delta_{z_n} (f) = \lim_{\ef} \zeta_{D^- \pl z_n, r_n \pr} (f)$. Suppose next that $C \notin \ef$, and take any   $C'  \in \ef$ with $C' \subset \en \setminus C$. Since $\prod_{n \in C'} \va z_n \vb \ge 1/\vc f \vd >0$ and $(z_n)_{n \in C'}$ is not regular, we can assume that, for all $n \in C'$, there exists at least one $m \in C'$, $m \neq n$, with $\va z_m \vb = \va z_n \vb$.
Then
fix a zero $u_n$ of $f$ in each $D^- \pl z_n, r_n \pr$ for all $n \in C'$. It is clear that, for $n \in C'$,
$$\xi_{D^+ \pl z_n, r_n \pr} (f) \le \prod_{\substack{\va u_m \vb = \va z_n \vb \\   m \in   C' \\ m \neq n}} \va z_n - u_m \vb = \prod_{\substack{\va z_m \vb = \va z_n \vb \\ m \in C' \\ m \neq n}} \va z_n - z_m \vb .$$
Since $\pl z_n \pr_{C'}$ is not regular, $\inf_{n \in C'} \xi_{D^+ \pl z_n, r_n \pr} (f) =0$, and we easily deduce from Corollary~\ref{conillosis} that $\lim_{\ef} \zeta_{D^+ \pl z_n, r_n \pr} (f) = \lim_{\ef} \xi_{D^+ \pl z_n, r_n \pr} (f) =0$. Finally, since $\delta_{z_n} \le \zeta_{D^+ \pl z_n, r_n \pr}$, we
 conclude that  $\lim_{\ef} \delta_{z_n} (f) =0$.
\end{proof}

\begin{ex}\label{larifa}
Let $\ze$ be a   sequence in $\de$ with $\prod_{n =1}^{\infty} \va z_n \vb >0$. Let $\tl R_i :i \in \en \tr$ be the set of the absolute values of all $z_n$, and suppose   that       $S_i := \va z_n - z_m \vb$ is constant for all $n, m \in \en$ with $\va z_n \vb =R_i= \va z_m \vb$, $i \in \en$.
Suppose also that $\ef$ is a nonprincipal ultrafilter in
$\en$ such that $\ze $ is not regular with respect to $\ef$.   Then
Proposition~\ref{gertrudisrivota} gives us
\begin{equation}\label{resorte}
\lim_{\ef} \delta_{z_n} = \lim_{\ef} \zeta_{D^+ (z_n, S_i)} .
\end{equation}
Obviously, we can define a map $\pi : \en \ra \en$ associating each $n$ with the number $\pi (n)$ with
 $\va z_n \vb = R_{\pi (n)}$. The meaning of $\pi (A)$ for $A \subset \en$ is clear, as well as that of
 $\pi (\ef)$. In fact, $\pi (\ef)$ is also
 a nonprincipal ultrafilter in $\en$. Now, it is easy to check that, by Equality~\ref{resorte}, if each $w_k $ is any point in $D^+ \pl z_n, S_{\pi(n)} \pr$,
 then
 $$\lim_{\ef} \delta_{z_n}   = \lim_{\pi(\ef)} \zeta_{D^+ \pl w_k, S_k \pr} .$$
\end{ex}

The following example
shows that the result in Proposition~\ref{lucile} cannot be sharpened (see  Remarks~\ref{xuxo} and~\ref{susichon}).

\begin{ex}\label{sandaliuca}
We consider $M$, $(R_i)$,  $(M_i)$, $\ze$ and $\ef$ to be the same as in Example~\ref{noterrias}. Suppose that $(N_k)$ is a sequence in $(0,1)$ with $\prod_{k=1}^{\infty} N_k >0$, where $N_1:= M^2$.
Clearly, we can find a sequence
$(A_k)$  in $\ef$ with $A_1 = \en$ and  $A_{k+1} \subsetneq A_k$ for all $k$ such that each $A_k$ satisfies the following property:
Given $i \in \en$, if  the cardinal $K_k^i$ of the $n $ in $A_k$ with  $z_n \in  C(0,  R_i)$ is not $0$, then
 $K_k^i \ge 2$ and
$$\pl \sqrt[M_i -1]{M} \pr^{K_k^i} \ge N_k .$$

 Now select  sequences $(r_n)$ and $(\delta_n)$ in $(0,1)$ with $\lim_{n \ra \infty} {r_n} =1$ and $\lim_{n \ra \infty} \delta_n =0$, and such that $0 < \delta_n < r_n \le \sqrt[M_i -1]{M}$ whenever $\va z_n \vb = R_i$. Next
consider a function $f \in \hac$ having exactly  $\cer_n$ simple zeros  in each $D^+ (z_n, \delta_n)$, where $\cer_n := \max \tl k \in \en : n \in A_k \tr$, and no other zeros in the corresponding $C(0, R_i)$ (see Proposition~\ref{gatillo}). Note that, for $i \in \en$, $\sum_{\va z_n \vb = R_i} \cer_n \le \sum_{k=1}^{\infty} K_k^i$. Consequently, for each $n \in \en$ with $\va z_n \vb =R_i$,
$\xi_{D^+ (z_n , r_n)} (f) = {r_n}^{\cer_n} \lambda_n $, where
$$\lambda_n   := \prod_{\substack{ \va z_m \vb = \va z_n \vb  \\ m \neq n}} \va z_n - z_m \vb^{\cer_m} \ge  \pl \sqrt[M_i -1]{M} \pr^{K_1^i + K_2^i +  K_3^i + \cdots } \ge \prod_{k=1}^{\infty} N_k   .$$

Note also that there exists a sequence $(l_n)$ in $\en$ with $\lim_{n \ra \infty} {r_n}^{l_n} =1/2$, and this sequence satisfies $\lim_{n \ra \infty} l_n = + \infty$. As in the proof of Proposition~\ref{lucile} (with $k_n =1$ for all $n$), we   see that, if $\lim_{\ef} \cer_n /l_n >0$,
 then $\lim_{\ef} \lambda_n =0$.
   Since this is not the case, we deduce that
 $\lim_{\ef} \cer_n /l_n =0$, and consequently that $\lim_{\ef} {r_n}^{\cer_n} =1$. By Corollary~\ref{conillosis}, taking into account that $\xi_{D^+ (z_n, r_n)} (f) \ge  {r_n}^{\cer_n} \prod_{k=1}^{\infty} N_k$ for all $n$,  we conclude   that $\lim_{\ef} \zeta_{D^+ (z_n, r_n)} (f) \neq 0$. On the other hand, since $A_k \in \ef$ for all $k$,  $\lim_{\ef} \cer_n = + \infty$, which implies that
for all $s \in (0,1)$, $\lim_{\ef} s^{\cer_n}  =0$ and consequently, $ \zeta_{\ze, \ef}^{\mathbf{1}, s} (f) = 0$.
 Thus, $\zeta_{\ze, \ef}^{\mathbf{1}, 1} \neq \lim_{\ef} \zeta_{D^+ (z_n, r_n)}$
 (see Remark~\ref{xuxo}).

On the other hand, Example~\ref{noterrias} tells us that, if the sequence $(r_n)$ is taken as above, then $\lim_{\ef} \zeta_{D^+ (z_n, r_n)} = \lim \zeta_{D^+ (z_n, s_n)}$, where $s_n = \sqrt[M_i -1]{M}$ whenever $\va z_n \vb = R_i$. Now, it is easy to see that $\lim \zeta_{D^+ (z_n, s_n)} = \lim_{\pi (\ef)} \zeta_{D^+ \pl w_i, \sqrt[M_i -1]{M} \pr}$, where each $w_i $ belongs to $D^+ \pl z_n , \sqrt[M_i -1]{M} \pr$ and $\pi$ is defined as in Example~\ref{larifa}. In other words, if $\we = (w_i)$ and $\mathbf{M} = \pl M_i -1 \pr$, then
$\lim_{\ef} \zeta_{D^+ (z_n, r_n)} = \zeta_{\we, \pi (\ef)}^{\mathbf{M}, M}$.

 We can prove that, for the function $f$ above,  if $t \in \pl 0, \zeta_{\we, \pi (\ef)}^{\mathbf{M}, M} (f) \pr$, then there is a sequence $(t_n)$ such that $\lim_{\ef} \zeta_{D^+ (z_n , t_n)} (f) = t$ (for this fact, see the proof of Theorem~\ref{liberban} in Section~\ref{davai}). It is also clear that $t_n \le \sqrt[M_i -1]{M}$ whenever $\va z_n \vb = R_i$. On the one hand, this implies that, if we put  $\phi_t := \lim_{\ef} \zeta_{D^+ (z_n , t_n)}$, then
 $$  \zeta_{\ze, \ef}^{\mathbf{1}, 1} \le   \phi_{t'} \le \phi_{t''} \le \zeta_{\we, \pi (\ef)}^{\mathbf{M}, M} $$
 and $\phi_{t'}(f) < \phi_{t''} (f)$ whenever $t' < {t''}$. This means by Proposition~\ref{lucile}  that there is no
set $\tl n_k : k \in \en \tr \in \ef$ such that $\lim_{k \ra \infty} t_{n_k} =1$. Since obviously $\lim_{\ef} t_n =1$, we see that Remark~\ref{susichon} is correct.
\end{ex}

\begin{proof}[Proof of Theorem~\ref{nadabouche}]
It is obvious that each $\delta_{\we, \af}$ belongs to $\mil$, because it can be written as $\lim_{\af} \varphi_{D^+ \pl w_m, 1/(2m) \pr}$. On the other hand, we take $\varphi = \lim_{\ef} \varphi_{D^+ \pl z_n , r_n \pr}$, and assume that $\lim_{\ef} r_n >0$. By Corollary~\ref{byemarcal}, we can assume that all the disks $D^+ \pl z_n , r_n \pr$ are pairwise disjoint and that  $r_n < \va z_n \vb$ for every $n $. We see that that the result follows from Proposition~\ref{gertrudisrivota} if $\ze$ is not regular with respect to $\ef$.

More in general,
by Corollary~\ref{pernia}, each $r_n$ can be taken in $\va \ka^{\times} \vb$. Now, for each $n \in \en$, pick $N_n \in \en$ with $   N_n \ge n +1$ and such that $\lim_{n \ra \infty} {r_n}^{N_n} =0$. Also, consider $A_n := \tl w_1^n, \ldots, w_{N_n}^n \tr \subset C \pl z_n , r_n \pr$ with $\va w_i^n - w_j^n \vb = r_n$ whenever $i \neq j$. We clearly see that all the $A_n$ can be taken in such a way that $D^+ (z, r_n) \cap D^+ (w, r_m) = \emptyset$ whenever $z \in A_n$ and $w \in A_m$. Using the lexicographic  order, define a sequence $\we$ with all the points in $\bigcup_{n =1}^{\infty} A_n$   (that is, if $m < m'$, then $w_m = w_i^n$ and $w_{m'} = w_j^{n'}$ with $n \le n'$ and, for $n =n'$, $i < j$).

Next consider the family $\mathfrak{F}$
 of the complements of all sets $D$ in $\en$ with the property
that
$$\lim_{\ef} \frac{\ca (\tl w_m : m \in D \tr  \cap A_n)}{N_n} =0.$$
It is a routine matter to check that
$\mathfrak{F}$ is a filter in $\en$ and that, given an ultrafilter $\af$ containing $\mathfrak{F}$, $\we$ is not regular with respect to $\af$.

It is also clear that,  if $s_m := r_n$ whenever $w_m \in A_n$, then  $\varphi =  \lim_{\af} \zeta_{D^+ \pl w_m , s_m \pr}$.
By Proposition~\ref{gertrudisrivota},  $\varphi = \lim_{\af} \delta_{w_m}$.
\end{proof}

We easily see that a slight modification of the above proof  shows that each $\delta_{\ze, \ef}$ with $\ze$ regular with respect to $\ef$ can be written as $\delta_{\we, \af}$ with $\we$ {\em not} regular with respect to $\af$.

\section{Kernels of seminorms}\label{davai}

In this section we prove all the results stated in Section~\ref{noporu}, but Theorem~\ref{zerolari}, which is proved in Section~\ref{ceropombo}.

\begin{proof}[Proof of Theorem~\ref{liberban}]
 Suppose that  $\varphi = \lim_{\ef} \zeta_{D^+ \pl z_n , s_n \pr}$, where $s_n < \va z_n \vb$ for all $n$. By Corollary~\ref{conillosis}, $ \lim_{\ef} \xi_{D^+ (z_n, s_n)} (f) =0$, so $\xi_{D^+ (z_n, s_n)} (f) <  r/ \pl 2 \vc f \vd \pr$  for all $n$ in some $C_0 \in \ef$.

Fix $n \in C_0$ and suppose that
 $w_1, \ldots, w_k$ are the zeros
  of $f$ in $C \pl 0, \va z_n \vb \pr$.
 It is clear that  the function
$
F_n: \ql 0, \va z_n \vb \qr \ra \er$  given by $
s  \mapsto  \prod_{j=1}^{k} \max \tl s , \va z_n - w_j \vb \tr ,
$
is continuous and increasing. Also  $F_n \pl \va z_n \vb \pr  = {\va z_n \vb}^{\cer \pl f,C \pl 0, \va z_n \vb \pr \pr }$,  and, consequently
$\lim_{n \in C_0} F_n \pl \va z_n \vb \pr =1$, and there exists $n_r \in C_0$ such that $F_n \pl \va z_n \vb \pr > r/ \vc f \vd$ for all $n \in C_0 $ with $n \ge n_r$.  Since $F_n (0) \le \xi_{D^+ (z_n, s_n)} (f) < r/ \vc f \vd$,    for $n \in C_0$ with $n \ge n_r$, we can find  $r_n \in (0 , \va z_n \vb)$
 such that
$$ \frac{r}{\vc f \vd} = F_n (r_n) = {r_n}^{\cer \pl f,D^+ \pl z_n, r_n \pr \pr} \prod_{ \va z_n - w_j \vb > r_n} \va z_n - w_j \vb .$$
Obviously
$\xi_{D^+ (z_n , r_n)} (f) = r/ \vc f \vd$
for all $n $. Consequently, if we define $\psi:= \lim_{\ef} \zeta_{D^+ (z_n , r_n)} $,
then by Corollary~\ref{conillosis},  $\psi (f) =r$.

 Note that any two of the above disks $D^+ \pl z_n , r_n \pr$ are either equal or disjoint. For each $k \in C_0$, we set
$n_k := \min \tl n : D^+ \pl z_n, r_n \pr = D^+ \pl z_k, r_k \pr \tr$, in such a way that the disks $D^+ \pl z_{n_k}, r_{n_k} \pr$ are pairwise disjoint.   Put $v_k := z_{n_k}$ and $t_k := r_{n_k}$ for all $k$.
Then define a new ultrafilter $\af$ in $\en$: A set $C \subset \en$ belongs to $\af$ if  the set
of all $ n \in C_0$ such that $
D^+ \pl z_n, r_n \pr = D^+ \pl v_k, t_k \pr$, for some $ k \in C  $, belongs to $\ef$. It is a routine matter to check that $\psi = \lim_{\af} \zeta_{D^+ (v_k , t_k)}$. On the other hand, by the definition of $r_n$, we easily see that each $\cer \pl f,D^+ \pl z_n, r_n \pr \pr \ge 1$, which implies that, for $k \in \en $ fixed,
$$\prod_{\substack{\va v_l \vb = \va v_k \vb \\ l \neq k}} \va v_k - v_l \vb \ge  \prod_{ \va z_{n_k} - w_m \vb > r_{n_k}} \va z_{n_k} - w_m \vb \ge \frac{r}{\vc f \vd} .$$ The fact that $\ve  $ is regular with respect to $\af$ follows easily and, consequently, $\psi$ belongs to $\mim$.

 On the other hand, by Proposition~\ref{gatillo},
we can find $g \in \hac$ with as many zeros in each $D^+ (z_n, r_n)$ as we need so that
$\psi (g) = 0$. This shows that $\psi$ is not a norm.
\end{proof}

\begin{prop}\label{diegosoper}
Let $\ze $ be a regular sequence with respect to $\ef \in \beta \en \setminus \en  $, and let $\kai  \in \com_{\ef}  (\ze)$.
Then there exists $f \in \hac$   with $\vc f  \vd =1$  such that
$$0 < \zeta_{\ze, \ef}^{\kai, r} (f)   \le r$$
for all $r \in (0,1)$
and $\zeta_{\ze, \ef}^{\kai, r} (f) < \zeta_{\ze, \ef}^{\kai, s} (f)$ if $0 < r < s < 1$.
\end{prop}

\begin{proof}
We consider $C \in \ef$ such that
$$M: = \inf_{n \in C}   \prod_{\substack{m \in C \\  m \neq n}} \va z_n - z_m \vb^{k_m} >0. $$
For $r \in (0,1)$ and $n \in C$, put $r_n := \sqrt[k_n]{r}$.  Consider  a sequence
    $(\delta_n)$  of positive numbers  converging to $0$
    with the property that
the disks $D^+ (z_n , \delta_n )$ are
pairwise disjoint.
   Then, since $\prod_{n \in C} \va z_n \vb^{k_n} >0$, we can use Proposition~\ref{gatillo} and take
$f \in \hac$  with $\vc f \vd =1$ and  $f (0) \neq 0$  having exactly $k_n$ simple zeros
in each $D^+ (z_n , \delta_n)$ whenever $n \in C$, and no other zeros in the circles $C(0, \va z_m \vb)$.

We put, for each $n \in C$, $T_n := \sum_{\substack{\va z_n - z_m \vb \le r_n \\ m \in C}} k_m $.
Note that if
$T := \inf_{n \in C} {r_n}^{T_n} =0$, then
$M    =0 $,
against our hypothesis. Thus $T >0$ and
$$\alpha := \lim_{\ef} \frac{T_n}{k_n} \in (1, +\infty).$$ On the other hand,
it is clear that, for every $n \in C$,
$$\xi_{D^+ (z_n, r_n)} (f) = {r_n}^{T_n} \prod_{ \substack{   \va z_m - z_n \vb > r_n \\ \va z_m \vb = \va z_n \vb \\ m \in C}} \va z_n - z_m \vb^{k_m}  $$
belongs to the interval
$\ql M  {r_n}^{T_n}  , {r_n}^{k_n}   \qr \subset \ql M T , r \qr$ and, by Corollary~\ref{conillosis}, $M T   \le \zeta_{\ze, \ef}^{\kai, r} (f) \le r  $.

Suppose next that $s \in (r,1)$ and $s_n := \sqrt[k_n]{s}$ for   $n \in C$.
As above,
$$\xi_{D^+ (z_n, s_n)} (f) = {s_n}^{S_n} \prod_{\substack{   \va z_m - z_n \vb > s_n \\ \va z_m \vb = \va z_n \vb \\ m \in C}} \va z_n - z_m \vb^{k_m}  $$
where $S_n := \sum_{\substack{\va z_n - z_m \vb \le s_n \\ m \in C}} k_m$.
Also, for all $n \in C$,
$$ {s_n}^{S_n }   \ge {s_n}^{T_n } \prod_{\substack{  s_n \ge \va z_m - z_n \vb > r_n \\ m \in C}} \va z_n - z_m \vb^{k_m}
\ge {r_n}^{T_n} \prod_{ \substack{  s_n \ge \va z_m - z_n \vb > r_n \\ m \in C}} \va z_n - z_m \vb^{k_m}  $$
and, consequently, the fact that $\zeta_{\ze, \ef}^{\kai, r} (f) = \zeta_{\ze, \ef}^{\kai, s} (f)$ implies that
$\lim_{\ef}   {s_n}^{T_n} = \lim_{\ef}   {r_n}^{T_n}$, that is, $s^{\alpha}  = r^{\alpha}$. We conclude that
$\zeta_{\ze, \ef}^{\kai, r} (f) < \zeta_{\ze, \ef}^{\kai, s} (f)$.
\end{proof}

\begin{rem}
In the proof of Proposition~\ref{diegosoper}, we see that if the set $C$ can be taken equal to $\en$, then
the same function $f$ makes the result hold for all $\ef \in \beta \en \setminus \en$ simultaneously.
\end{rem}

Prior to proving Theorem \ref{thesidubar}, we give the following lemma.

\begin{lem}\label{siatillo}
Let $\alpha : (0,1) \ra [0 , + \infty ]$ be an increasing function. If $r_0 \in (0,1)$, then there exist $r_1 >r_0$ and $M \in \er$ such that
$$
 \va r^{\alpha (\max \tl  r_0 , r \tr )} - {r_0}^{\alpha (\max \tl  r_0 , r \tr )} \vb \le
   M  \va r - r_0 \vb
$$
for every $r \in (r_0/2, r_1]$.
\end{lem}

\begin{proof}
Let $ \beta := \inf_{r >r_0} \alpha (r)$.
  If $\beta = + \infty$,   then $\alpha (r) = + \infty$ whenever $r \in (r_0,1)$, so
$r^{\alpha (r)} - {r_0}^{\alpha (r)}  =0$. If $\beta < + \infty$, we find $r_1 > r_0$ such that $\beta   \le \alpha (r_1 ) < + \infty$. By the Mean Value Theorem,
for each $r \in (r_0, r_1]$, there exists $c \in (r_0, r)$ with ${r}^{\alpha (r)} - {r_0}^{\alpha (r)} = \alpha (r) \ c^{\alpha (r) -1} \pl r - r_0 \pr$. Now, if
$\beta <1$, then $r_1$ can be taken with $a (r_1) <1$, giving $c^{\alpha (r) -1} \le {r_0}^{\beta -1} $ and ${r}^{\alpha (r)} - {r_0}^{\alpha (r)} \le \alpha (r_1) \ {r_0}^{\beta -1} \pl r - r_0 \pr$. On the other hand, if $\beta \ge 1$, then $c^{\alpha (r) -1} \le 1$ and
 ${r}^{\alpha (r)} - {r_0}^{\alpha (r)} \le \alpha (r_1)   \pl r - r_0 \pr $.

We next consider the case $0 < r < r_0$. First, if $\alpha (r_0) = + \infty$, then
${r_0}^{\alpha (r_0)} - r^{\alpha (r_0)}  =0$. On the other hand, if  $\alpha (r_0) < + \infty$, then  there exists $c \in (r, r_0)$ with
${r_0}^{\alpha (r_0)} - r^{\alpha (r_0)} = \alpha (r_0) \ c^{\alpha (r_0) -1} \pl r_0 - r \pr $.
This implies that, when $\alpha (r_0) \ge 1$,
$${r_0}^{\alpha (r_0)} - {r}^{\alpha (r_0)}  \le \alpha (r_0)   \pl r_0 - r \pr $$
for all $r \in (0,r_0)$, whereas when  $\alpha (r_0) < 1$
$${r_0}^{\alpha (r_0)} - {r}^{\alpha (r_0)}  \le \alpha (r_0)  \ \pl \frac{r_0}{2} \pr^{\alpha (r_0) -1} \pl r_0 - r \pr $$
for $r \in (r_0/2,r_0)$.

The conclusion follows easily.
\end{proof}

\begin{proof}[Proof of Theorem \ref{thesidubar}]
We write $\zeta_r := \zeta_{\ze, \ef}^{\kai, r}$, for short.
We deduce from  Proposition~\ref{diegosoper}  that the map
$\Phi : (0, 1) \ra \ma$, $r \mapsto \zeta_r$, is injective.
Let us next see  that it is continuous.
Fix $f \in \hac$ with   $0 < \vc f \vd \le 1$  and,   for $0< r <1$ and $n \in \en$,  put $\cer_n (r) := \cer \pl f ,  D^+ \pl z_n, \sqrt[k_n]{r} \pr \pr$ and
$$\alpha (r) :=  \lim_{\ef} \frac{\cer_n (r)}{k_n} .$$
It is easy to see that the function $\alpha : (0,1) \ra [0 , + \infty ]$ is increasing.

Now, consider $0 <s <r <1$. Since there exists $C \in \ef$ such that   $\lim_{\substack{n \in C \\ n \ra \infty}} \va z_n \vb^{k_n} =1$ and we are dealing with an ultrafilter, there is no loss of generality if we assume that   $  \sqrt[k_n]{r} < \va z_n\vb$ for every $n \in C$.
By Lemma~\ref{pijamalari},
$$
\va \zeta_{D^+ \pl z_n, \sqrt[k_n]{r} \pr} (f) - \zeta_{D^+ \pl z_n, \sqrt[k_n]{s} \pr} (f) \vb \le
   \pl \sqrt[k_n]{r} \pr^{\cer_n (r)} - \pl \sqrt[k_n]{s} \pr^{\cer_n (r)}
  $$
 for all $n \in C$, so
 \begin{eqnarray*}
\va \zeta_r (f) - \zeta_s (f) \vb &\le&
  \lim_{\ef}  \pl \sqrt[k_n]{r} \pr^{\cer_n (r)} - \lim_{\ef}  \pl \sqrt[k_n]{s} \pr^{\cer_n (r)}  \\
 &=&   r^{\alpha (r) } - s^{\alpha (r)}   .
 \end{eqnarray*}

The fact that   $\Phi$ is
 continuous is now easy by Lemma~\ref{siatillo}.

Let us next study whether there exist $\lim_{r \ra 0} \zeta_r$ and $\lim_{r \ra 1} \zeta_r$. Note that, given $f \in \hac$, the map
$\Psi_f : (0, 1) \ra \er$, $r \mapsto \zeta_r (f)$ is increasing and bounded, so there exist
$\zeta_0 (f) := \lim_{r \ra 0} \Psi_f (r)$ and $\zeta_1 (f) := \lim_{r \ra 1} \Psi_f (r)$. It is clear that
 the maps $\zeta_0$ and $\zeta_1$ defined in this way belong to $\ma$. Also, since, $\zeta_r \neq \zeta_s$ for
 every $r \neq s$, we    conclude that the the natural extension of $\Phi$ to a new map
 (call it also $\Phi$) $\Phi :[0,1] \ra \ma$ is indeed injective and continuous, so it is a homeomorphism onto its image. The fact that $\Phi [0,1] = \cl \pl \zeta_{\ze, \ef}^{\kai, 0} , \zeta_{\ze, \ef}^{\kai, 1} \pr$ is now easy.
 \end{proof}

\begin{proof}[Proof of Corollary~\ref{pase}]
The fact that $\ker \zeta_{\ze, \ef}^{\kai, r} = \ker \zeta_{\ze, \ef}^{\kai, s}$, for $r, s \in (0,1)$,
 follows easily from Lemma~\ref{correicion} and Corollary~\ref{conillosis}. Also, if $r \in (0,1)$, then $ \zeta_{\ze, \ef}^{\kai, r} \le \zeta_{\ze, \ef}^{\kai, 1}$. Since   $\zeta_{\ze, \ef}^{\kai, 1} = \lim_{r \ra 1}  \zeta_{\ze, \ef}^{\kai, r}$,   $\ker \zeta_{\ze, \ef}^{\kai, 1} = \ker \zeta_{\ze, \ef}^{\kai, r} $ for all $r \in (0,1)$. 
 \end{proof}

Corollary~\ref{pase} has a natural extension. Note that, if for some fixed $r \in (0,1)$, $\we $ satisfies $\va \we - \ze \vb_{\ef}^{\kai} \le r$ (meaning that $\va w_n - z_n \vb \le \sqrt[k_n]{r}$ for all $n$ in some $C \in \ef$),
then $\zeta_{\ze, \ef}^{\kai, s} = \zeta_{\we, \ef}^{\kai, s}$ for all $s \in [r, 1]$. This implies that $\zeta_{\ze, \ef}^{\kai, r}$
ramifies into infinitely many segments
$\pl \zeta_{\we, \ef}^{\kai, 0} , \zeta_{\ze, \ef}^{\kai, r} \qr $.  In the following corollary, we see that all the seminorms in the union of
these segments have the same kernel. We write  $\va \we - \ze \vb_{\ef}^{\kai} = r$ to indicate that $\va w_n - z_n \vb = \sqrt[k_n]{r}$ for all $n$ in some $C \in \ef$. We also write $\myt := \bigcup_{  0 < r <1 }  \tl \we : \va \we - \ze \vb_{\ef}^{\kai} \le r \tr$ and  $\mut := \bigcup_{ \we  \in \myt} \pl \zeta_{\we, \ef}^{\kai, 0} , \zeta_{\we, \ef}^{\kai, 1} \qr $.

\begin{cor}\label{lz}
Let
$\ze$ be a regular sequence with respect to $\ef \in \beta \en \setminus \en$. Then, for each $\kai \in  \com_{\ef}  (\ze) $, all the seminorms in $\mut$ have the same kernel. Moreover,  for every $\varphi  \in  \mut $, $\ker \varphi$  is strictly contained in $\ker \zeta_{\we, \ef}^{\kai, 0} $ whenever $ \we \in \myt$, and if $r \in  \va \ka^{\times} \vb  \cap (0,1)$, then $$\ker \varphi   = \bigcap_{\va \we - \ze \vb_{\ef}^{\kai} = r } \ker \zeta_{\we, \ef}^{\kai, 0} =  \bigcap_{ \va \we - \ze \vb_{\ef}^{\kai} = r  } \ker \delta_{\we, \ef} .$$
\end{cor}

 \begin{proof}
 If $r \in (0,1)$ and $\va \we - \ze \vb_{\ef}^{\kai} \le r$, then $  \zeta_{\we, \ef}^{\kai, r} = \zeta_{\ze, \ef}^{\kai, r}$ and, by Corollary~\ref{pase},      $\ker  \zeta_{\we, \ef}^{\kai, r} = \ker \zeta_{\ze, \ef}^{\kai, 1}  $.

 On the other hand, we fix $\varphi = \zeta_{\ve, \ef}^{\kai, s} \in \mut$,  with $s \in (0,1]$ and $\va \ve - \ze \vb _{\ef}^{\kai} \le s'$, $s' \in (0,1)$. Clearly, given $r \in (0,1)$ and a    sequence $\we$ with $\va \we - \ze \vb_{\ef}^{\kai} \le r$, if $t:= \max  \tl r, s' \tr$,
 $ \zeta_{\we, \ef}^{\kai, 0} \le \zeta_{\we, \ef}^{\kai, t} = \zeta_{\ve, \ef}^{\kai, t}$,
 so $\ker \varphi \subset \ker \zeta_{\we, \ef}^{\kai, 0} \subset \ker \delta_{\we, \ef} $. Now, the strict inclusion $  \ker \varphi =\ker \zeta_{\we, \ef}^{\kai, t} \subsetneq \ker \zeta_{\we, \ef}^{\kai, 0} $ follows easily from Proposition~\ref{diegosoper} and the definition of $\zeta_{\we, \ef}^{\kai, 0} $.

  Suppose next that $r \in \va \ka^{\times} \vb$, and that $f \in \hac$ satisfies $\varphi (f) \neq 0$. Thus,
 $\zeta_{\ze, \ef}^{\kai, r} (f) \neq 0$.
 It is clear that,    for each $n$, there exists $w_n \in \ka$ with $\va w_n - z_n \vb= \sqrt[k_n]{r}$
 such that there are no zeros of $f$ in $D^- \pl w_n , \sqrt[k_n]{r} \pr$, meaning  that $\va f(w_n) \vb = \zeta_{D^+ \pl z_n , \sqrt[k_n]{r}\pr} (f) $
 (see Corollary~\ref{angiosperma}), and  $\va f(w_n) \vb = \zeta_{D^+ \pl w_n , \sqrt[k_n]{t}\pr} (f) $ for all  $t \in (0,r)$.
Consequently, for the corresponding sequence $\we$,
 $$\delta_{\we, \ef} (f) = \zeta_{\we, \ef}^{\kai, 0} (f) = \zeta_{\ze, \ef}^{\kai, r} (f) \neq 0.$$
The conclusion    follows immediately.
\end{proof}

Now, if $Z_{\ze, \ef} := \bigcup_{\kai \in \com_{\ef} (\ze)} \myt$, it is clear that $Z_{\ze, \ef} = Z_{\we, \ef}$ whenever $\we \in Z_{\ze, \ef}$. Also, for each $\kai \in \com_{\ef} (\ze)$  the family $\mathcal{F}_{\ef}^{\kai} := \tl Z_{\we, \ef}^{\kai} : \we \in Z_{\ze, \ef} \tr$  is a partition of  $Z_{\ze, \ef} $. It is obvious that if $0 < \lim_{\ef} l_n /k_n < + \infty$, then $Z_{\we, \ef}^{\eli} = Z_{\we, \ef}^{\kai}$ and $A_{\we, \ef}^{\eli} =A_{\we, \ef}^{\kai}$. Note that, by Corollary~\ref{zerote}, if $\lim_{\ef} l_n / k_n =0$, then $\ker \varphi \subsetneq \ker \psi$ whenever $\varphi \in A_{\we, \ef}^{\kai}$ and $\psi \in  A_{\we, \ef}^{\eli}$.

\begin{proof}[Proof of Corollary~\ref{zerote}]
Obviously, for every $r   \in (0,1)$, $\zeta_{\ze, \ef}^{\eli, r} \le \zeta_{\ze, \ef}^{\kai, 0} $, so $ \ker \zeta_{\ze, \ef}^{\kai, 0} \subset \ker \zeta_{\ze, \ef}^{\eli, 1}$. Now, the conclusion follows from Corollary~\ref{lz}.
\end{proof}

\begin{proof}[Proof of Corollary~\ref{lluc}]
Fix $\varphi = \ker \zeta_{\ze, \ef}^{\kai, r} \in \mak$. By Corollary~\ref{lz}, $\ker \zeta_{\ze, \ef}^{\kai, 0}$
 strictly contains   $\ker \varphi$, so $\ker \varphi$ is not maximal.
On the other hand, by Remark~\ref{honble} (assumming without loss of generality that $C= \en$), we   fix $r_0 \in (0,1)$ such that  all the  disks $D^+ \pl z_i , \sqrt[k_i]{r_0} \pr$ are pairwise disjoint. Next, taking into account that $\prod_{n=1}^{\infty} \va z_n \vb^{k_n} >0$, it is easy to see that there exists a sequence $(l_n)$ in $\en$ with $\lim_{n \ra \infty } l_n = + \infty$ such that $\prod_{n=1}^{\infty} \va z_n \vb^{l_n k_n} >0$. Now, we can use Proposition~\ref{gatillo} to construct $f \in \hac$ having $l_n k_n$ zeros in each $D^+ \pl z_n, \sqrt[k_n]{r_0} \pr$. Obviously, $\zeta_{\ze, \ef}^{\kai, r_0} (f) =0$. By Corollary~\ref{pase},   $\ker \varphi = \ker \zeta_{\ze, \ef}^{\kai, r_0}$, so $\varphi$ is not a norm.
\end{proof}

\begin{proof}[Proof of Corollary~\ref{lucia}]
Clearly, if $\lim_{\ef} k_n = + \infty$,
 then $\zeta_{\ze, \ef}^{\mathbf{1}, 1/2} \le \zeta_{\ze, \ef}^{\kai, 0}$,
 and
$     \ker \zeta_{\ze, \ef}^{\kai, 1/2} \subsetneq \ker \zeta_{\ze, \ef}^{\kai, 0} \subset \ker \zeta_{\ze, \ef}^{\mathbf{1}, 1/2} . $
It follows from Corollary~\ref{lluc} that $\ker \zeta_{\ze, \ef}^{\kai, 0}$ is nonzero and nonmaximal. The converse is easy.
\end{proof}

\begin{proof}[Proof of Theorem~\ref{kubzdelamagdalena}]
Let $\ql \delta_{\ze, \ef} , \vc \ \vd \qr_{\mil} $ be the family of all seminorms in $\mil$ of the form $\lim_{\ef} \zeta_{D^+ \pl z_n , r_n \pr}$. It is immediate to see that $\ql \delta_{\ze, \ef} , \vc \ \vd \qr_{\mil}$ is linearly ordered with respect to the usual order $\le$. We next prove that $A_{\ze}^{\ef}:= \cl \ql \delta_{\ze, \ef} , \vc \ \vd \qr_{\mil} $ is also linearly ordered.

 Given different $\varphi_1, \varphi_2 \in   A_{\ze}^{\ef}$, there exists $f \in \hac$ which separates them. We can assume without loss of generality that $\varphi_1 (f) < \varphi_2 (f)$. Next, as in the proof of Theorem~\ref{liberban}, for $  r \in \pl \varphi_1 (f) , \varphi_2 (f) \pr$ we can find $r_n \in (0,1)$ such that $\zeta_{D^+ \pl z_n ,r_n \pr } (f) = r$ for all $n$ in a certain $C \in \ef$. Obviously $\psi := \lim_{\ef} \zeta_{D^+ \pl z_n , r_n \pr} \in  \ql \delta_{\ze, \ef} , \vc \ \vd \qr_{\mil}$ satisfies
$$ \varphi_1 (f)   < \psi (f)   < \varphi_2 (f)  .$$

Now, let $\pl \varphi_{\ze, \ef}^{\kai^{\lambda}, r_{\lambda}} \pr_{\lambda \in \Lambda}$ be a net   in $ \ql \delta_{\ze, \ef} , \vc \ \vd \qr_{\mil}$ converging to $\varphi_1$. Then there exists $\lambda_0 \in \Lambda$ such that $\varphi_{\ze, \ef}^{\kai^{\lambda }, r_{\lambda}} (f) < \psi (f)$ for all $\lambda \ge \lambda_0$, $\lambda \in \Lambda$. In particular, for each $\lambda \ge \lambda_0$,
$$\lim_{\ef} \zeta_{D^+ \pl z_n , \sqrt[{k^{\lambda}}_n]{r_{\lambda}} \pr } (f) < \lim_{\ef} \zeta_{D^+ \pl z_n , r_n \pr} (f) ,$$
and consequently there exists   $E_{\lambda} \in \ef$ such that
$${r_{\lambda}}^{1 / {k^{\lambda}}_n } <  r_n$$
for all $n \in E_{\lambda}$. This obviously implies that, for $g \in \hac$,  $\varphi_{\ze, \ef}^{\kai^{\lambda}, r_{\lambda}} (g) \le \psi (g)$ whenever $\lambda \ge \lambda_0$. We conclude that $\varphi_1 \le \psi$. Similarly $ \psi  \le \varphi_2$. The fact that the compact set $   A_{\ze}^{\ef}$ is linearly ordered follows.

We next see that $ A_{\ze}^{\ef}$ is connected. Suppose to the contrary that $ A_{\ze}^{\ef}$ is the union of two disjoint (nonempty) clopen subsets $U, V$  (with respect to the induced topology).
Suppose also that $\varphi_1 \in U$ and $\varphi_2 \in V$ satisfy $\varphi_1 \le \varphi_2$.
We define
$$\psi_1 := \sup \tl \varphi \in U   : \varphi \le \varphi_2 \tr .  $$
Obviously $\psi_1 \in U$ and $\psi_1 \le \varphi_2$. Similarly,
$$\psi_2 := \inf \tl \varphi \in V   : \psi_1 \le \varphi  \tr    $$
belongs to $V$, and $\psi_1 \le \psi_2$.
 As we showed above there exists $\psi \in A_{\ze}^{\ef}$, different from $\psi_1$ and $\psi_2$ such that $\psi_1 \le \psi \le \psi_2$. It is clear that $\psi  \notin U \cup V$, which is impossible.

Now suppose that $\varphi \in  A_{\ze}^{\ef}$, $\varphi \neq \delta_{\ze, \ef}, \vc \ \vd $. Then there exists $r \in (0,1)$
such that $\zeta_{\ze, \ef}^{\mathbf{1}, r} \le \varphi$ and, by Corollary~\ref{lluc}, $\ker \varphi $ is not maximal. On the other hand, since $\varphi \neq \vc \ \vd$,  $\ker \varphi \neq \tl 0 \tr$, as follows from Proposition~\ref{josejrut}.
\end{proof}

\begin{proof}[Proof of Theorem~\ref{juntafaculta}]
 Suppose that
$\varphi \notin \muz$, that is,  for all $C \in \ef$,   $ \inf_{n \in C} \prod_{\substack{    m \in C \\  m \neq n }} \va z_n - z_m \vb^{k_m} =0 $.
We deduce that, if $\emi \in \com_{\ef} (\ze)$, then $\lim_{\ef} m_n /k_n =0$, and $\zeta_{\ze, \ef}^{\emi,1} \le \varphi_{\ze, \ef}^{\kai, r}$. Thus, $\sup_{\emi \in \com_{\ef} (\ze)}  \zeta_{\ze, \ef}^{\emi,1} \le \varphi$.

To finish the proof, it is enough to see that for each
  $f \in \hac$, there exists $\emi (f) \in \com_{\ef} (\ze)$  such that $\varphi (f) = \zeta_{\ze, \ef}^{\emi (f), 1} (f)$.
 Consider  $f \in \hac$.
If $\varphi (f) =0$, then $\zeta_{\ze, \ef}^{\mathbf{1}, r/2} (f) = 0$ and, by Corollary~\ref{pase},
$\zeta_{\ze, \ef}^{\mathbf{1}, 1} (f) = 0$, so we can take $\emi (f) = \mathbf{1}$.
Next, suppose that $f \notin \ker \varphi$.

For each $n \in \en$,   put $r_n = \sqrt[k_n]{r}$ and
$m_n := \cer \pl f, D^- \pl z_n , r_n \pr \pr $. If there exists $C \in \ef$ such that $m_n =0$ for all $n \in C$, then by Corollary~\ref{angiosperma} $\va f(z_n) \vb  = \zeta_{D^+ \pl z_n , r_n \pr } (f)$ for every $n \in C$. It follows easily that $\delta_{\ze, \ef} (f) = \zeta_{\ze, \ef}^{\mathbf{1}, M} (f) = \varphi (f)$ for all $M \in (0,1)$, so  $\varphi (f) = \zeta_{\ze, \ef}^{\mathbf{1}, 1} (f)$. On the other hand, if the above set $C$ does not belong to $\ef$, then for $n \in \en$
 $$  \prod_{\substack{\va z_j \vb = \va z_n \vb \\ j \in \en \setminus C \\ j \neq n}} \va z_n - z_j \vb^{m_j} \ge \xi_{D^+ \pl z_n , r_n \pr } (f) .$$
Also, by Corollary~\ref{conillosis},   there is $D \in \ef$ with $D \subset \en \setminus C$ such that $\xi_{D^+ \pl z_n , r_n \pr } (f) \ge \varphi (f) / 2 \vc f \vd$ for all $n \in D$.
Therefore
  $\emi (f) := \pl \max \tl m_n, 1 \tr \pr$ belongs to $\com_{\ef} (\ze)$.
On the other hand, it is a routine matter to check that, for   $M \in (0,1)$ fixed, the set of all $n$ with $\sqrt[m_n]{M} < r_n$ belongs to $\ef$ and, by Lemma~\ref{correicion},
$$M \xi_{D^+ \pl z_n, r_n \pr} (f) \le \xi_{D^+ \pl z_n, \sqrt[m_n]{M} \pr} (f) \le \xi_{D^+ \pl z_n, r_n \pr} (f) .$$
Again by Corollary~\ref{conillosis}, this implies that $M \varphi (f) \le \zeta_{\ze, \ef}^{\emi (f), M} (f) \le \varphi (f)$ for all $M \in (0,1)$,
and consequently $\varphi (f) = \zeta_{\ze, \ef}^{\emi (f), 1} (f)$.
\end{proof}

\begin{rem}\label{saradientes}
The following should be compared with Proposition~\ref{lucile}. Let $\ze$ be a regular sequence with respect to $\ef \in \beta \en \setminus \en$, and let $\kai$ be a  sequence in $\en$. Let $ r \in (0,1)$ be such that $z_m \notin D^+ \pl z_n, \sqrt[k_n]{r} \pr$ whenever $m \neq n$.
We see in the proof of Theorem~\ref{juntafaculta} that, if $\kai \notin \com_{\ef} (\ze)$, then
$\varphi_{\ze, \ef}^{\kai , s} = \varphi_{\ze, \ef}^{\kai , r}$ for all $s \in (0, r]$.
\end{rem}

\begin{proof}[Proof of Corollary~\ref{sussosoplador}]
Set $r_n := \sqrt[k_n]{r}$ for all $n$. Suppose that there exists $f \in \ker \varphi$, $f \neq 0$, and put $\cer_n := \cer \pl f, C \pl 0 , \va z_n \vb \pr \pr$ for all $n \in \en$.  By Corollary~\ref{conillosis}, $\lim_{\ef} \xi_{D^+ \pl z_n , r_n \pr} (f) =0$, and consequently $\lim_{\ef} {r_n}^{\cer_n} =0$. This implies that  $\lim_{\ef} \cer_n / k_n = + \infty$, so there exists $C \in \ef$ with
$ \cer_n \ge k_n$ for all $n \in C$. Since $\vc f \vd \ge 1/\prod_{n \in C} \va z_n \vb^{\cer_n}$, we conclude that $\prod_{n \in C} \va z_n \vb^{k_n} >0$.
Now the fact that $\kai$ belongs to $\com_{\ef} (\ze)$ is easy.

 On the other hand, if $\ker \varphi = \tl 0 \tr$, then the fact that $\varphi = \vc \ \vd$ follows from Proposition \ref{josejrut}.
  \end{proof}

\begin{cor}
Given $\varphi := \varphi_{\ze, \ef}^{\kai, r} \in   \mim$.
  If $\varphi \notin \mak$, then
$$\ker \varphi = \bigcap_{\emi \in \com_{\ef} (\ze)} \ker \zeta_{\ze, \ef}^{\emi, 1}     .$$
\end{cor}

\section{Seminorms $\zeta_{\ze, \ef}^{\kai, 0}$ and  $\zeta_{\ze, \ef}^{\kai, 1}$. A special kernel}\label{ceropombo}

As we mentioned before Corollary~\ref{lucia}, in the case when $\kai = \mathbf{1}$, $\zeta_{\ze, \ef}^{\kai, 0} = \delta_{\ze, \ef} \in \mil$. Proposition~\ref{tobien} says that seminorms $\zeta_{\ze, \ef}^{\kai, 0}$ and $\zeta_{\ze, \ef}^{\kai, 1}$ always belong to $\mil$.
We finish this section with a proof of Theorem~\ref{zerolari} and an example showing that it cannot be generalized to other seminorms defined as an infimum over a family in $\mak$.

\begin{prop}\label{tobien}
Let $\ze$ be a regular sequence with respect to $\ef \in \beta \en \setminus \en$, and let $\kai \in \com_{\ef} (\ze)$.  Then
$\zeta_{\ze, \ef}^{\kai, 0} , \zeta_{\ze, \ef}^{\kai, 1}  \in \mil$.
\end{prop}

\begin{proof}
We prove   that
there exist a sequence $\we$ in $\de$ with $\lim_{m \ra \infty} \va w_m \vb =1$, a nonprincipal ultrafilter $\af$ in $\en$, and sequences $(r_m)$ and $(t_m)$ in $(0,1)$ such that
$\zeta_{\ze, \ef}^{\kai, 0} = \lim_{\af} \zeta_{D^+ \pl w_m , r_m \pr }$ and $\zeta_{\ze, \ef}^{\kai, 1} = \lim_{\af} \zeta_{D^+ \pl w_m , t_m \pr }$.

We fix $s \in (0,1)$. For each $n \in \en$ and $j=1 , \ldots, n$, let $r_n^j := \sqrt[k_n]{s/j}$.   We write $A_n := \tl r_n^j : 1 \le j \le n \tr$, and consider $A:= \bigcup_{n=1}^{\infty} A_n$. Then rename the $r_n^j \in A$ by
$r_1 := r_1^1,  r_2 := r_2^1,  r_3 := r_2^2 ,  r_4 := r_3^1 ,\ldots$ We also put $w_m := z_n $ when $r_m = r_n^j$.

For each $N \in \en$, each $D \in \ef$, and each sequence $\eli$ in $\en$   such that $\lim_{\ef} l_n  =+ \infty$ and  $l_n \le n$ for all $n$, consider the set $D_{\eli}^N$ of all   $m \in \en$ satisfying $r_m = r_n^j$, for $N \le j \le l_n$ and $n \in D$.
It is easy to check that the family
$\mathcal{F}$
 of all sets $D_{\eli}^N$ is the basis for a filter $\mathfrak{F}$ in $\en$.
Fix an ultrafilter $\af$ containing $\mathfrak{F}$. Since, for $N \in \en$ fixed, the set $C$ of all $m$ such that $r_m = r_n^j$ with $j \ge N$ belongs to $\mathfrak{F}$,     $\zeta_{\ze, \ef}^{\kai, s/N} \ge \psi:= \lim_{\af} \zeta_{D^+ \pl w_m , r_m \pr }$, and consequently $\zeta_{\ze, \ef}^{\kai, 0} \ge \psi$.

On the other hand, for $f \in \hac$ and $\epsilon >0$, there exists $C \in \af$ such that $\zeta_{D^+ \pl w_m , r_m \pr } (f) < \psi (f) + \epsilon$ for all $m \in C$. Consider the set $M_n :=  \tl j : r_n^j =r_m, m \in C \tr$ for each $n \in \en$, and note that the family $D$ of all $n$ with $M_n \neq \emptyset$ belongs to $\ef$. Also, for $n \in D$, define $m_n := \min M_n$. By the construction of $\mathcal{F}$, $N := \lim_{\ef} m_n $ belongs to $\en$, and consequently the set of all $n \in D$ with $m_n = N$ belongs to $\ef$. Then $\zeta_{D^+ \pl z_n , r^N_n \pr } (f) < \psi (f) + \epsilon$ for all $n \in D$ and $\zeta_{\ze, \ef}^{\kai, s/N} (f) \le \psi (f)  + \epsilon$.
It is a routine matter  to check that $\zeta_{\ze, \ef}^{\kai, 0} \le \psi $.

As for $\zeta_{\ze, \ef}^{\kai, 1}$, we define $t_n^j := \sqrt[k_n]{1 - s/j}$ for each $n \in \en$ and $j=1 , \ldots, n$, and
set $t_m := t_n^j$ in a similar way as above.   Consider also the same sequence $(w_m)$ and the same ultrafilter $\af$ as above. The fact that $\zeta_{\ze, \ef}^{\kai, 1} = \lim_{\af} \zeta_{D^+ \pl w_m , t_m \pr }$ follows easily.
\end{proof}

\begin{lem}\label{tus}
Let $\ze  $ and $\we$ be regular sequences    with respect to $\ef $ and $ \af \in \beta \en \setminus \en$, respectively, and let $ \kai \in \com_{\ef}  (\ze)$  and $ \eli   \in
 \com_{\af}  (\we)$. Suppose that
 $\ker \zeta^{\kai, 0}_{\ze, \ef} \subset \ker \zeta^{\eli, 0}_{\we, \af}$.
Then, for each $C \in \ef$ and $r \in (0,1)$,  the set
 $$B_C^r := \tl m \in \en :   w_m \in \bigcup_{n \in C} D^+ \pl z_n , \sqrt[k_n]{r} \pr \tr$$
 belongs to $\af$.
 \end{lem}

 \begin{proof}
  Suppose to the contrary that a $B_{C'}^{t_0}$ does not belong to $\af$. By Remark~\ref{honble}, we can take $r_0 \in (0,1)$ such  that
 all the  disks $D^+ \pl z_i , \sqrt[k_i]{r_0} \pr$ are pairwise disjoint, for $i$ in a fixed set in $\ef$. We assume that $t_0 \le r_0$.
 Let $(\delta_n)$ be a sequence of positive numbers with $\delta_n < \sqrt[k_n]{t_0}/n$ for all $n$. Then take $g \in \hac$ with $g(0) =1$ having exactly $k_n$ simple zeros in each $D^+ (z_n, \delta_n)$ whenever $n \in C'$  and no other zeros in the circles $C(0, \va z_n \vb )$ with $n \in C'$, and having  no zeros in the circles   $C(0, \va w_m \vb )$ for which $\va w_m \vb \neq \va z_n \vb$ for all $n \in {C'}$ (see Proposition~\ref{gatillo} and Remark~\ref{kuku}).
 Fix $s_0 \in (0,1)$ and, for each $m \notin B_{C'}^{t_0}$,
 take $\eta_m \in \pl 0, \sqrt[l_m]{s_0 /m} \pr$
 such that $D^+ \pl w_m , \eta_m \pr \cap D^+ \pl z_n , \sqrt[k_n]{t_0} \pr = \emptyset$ for all $n \in C'$.
Obviously, if $\va w_m \vb \neq \va z_n \vb$ for all $n \in C'$, then $\xi_{D^+ \pl w_m , \eta_m \pr} (g) =1$. In any other case,
 define $N (m) \in C'$ to be the index corresponding to the smallest value in $\tl \va z_n - w_m \vb : n \in C', \va z_n \vb = \va w_m \vb \tr$. It is clear that
$$ \xi_{D^+ \pl w_m , \eta_m \pr} (g) = \prod_{\substack{\va z_n \vb = \va w_m \vb \\ n \in C'}}  \va z_n - w_m \vb^{k_n}  \ge
 t_0 \prod_{\substack{\va z_n \vb = \va w_m \vb \\ n \in C' \\ n \neq N (m)}} \va z_n - z_{N (m)} \vb^{k_n} . $$
 By Corollary~\ref{conillosis}, this implies  that $\zeta^{\eli, s}_{\we, \af} (g)\ge M t_0 \vc g \vd >0$ for all $s \in (0,1)$, and consequently $\zeta^{\eli, 0}_{\we, \af} (g) >0$.
 On the other hand,  as we easily deduce from the proof of Proposition~\ref{diegosoper}, $\zeta^{\kai, 0}_{\ze, \ef} (g) =0$, against our assumption.
\end{proof}

 \begin{proof}[Proof of Theorem~\ref{zerolari}]
 We  prove both parts separately.

\smallskip

{\em Proof of the second part.} We suppose that $\ker \zeta^{\kai, 0}_{\ze, \ef} = \ker \zeta^{\eli, 0}_{\we, \af}$.
By Remark~\ref{honble}, we can take $C_0 \in \ef$ and  $r_0  \in (0,1)$ such that, for $n \in C_0$, all the disks $D^+ \pl z_n , \sqrt[k_n]{r_0} \pr$ are pairwise disjoint. Similarly,
 we take $D_0 \in \af$ and $s_0 \in (0,1)$ such that all the $D^+ \pl w_m , \sqrt[l_m]{s_0} \pr$ are pairwise disjoint for $m \in D_0$. We first prove the following claim.

 \begin{claim}\label{arxona}
 There are $C_a \in \ef$ with $C_a \subset C_0$ and  $D_a \in \af$ with $D_a \subset D_0$  such that for each $n \in C_a$ there exists exactly one $m \in D_a$ with $\va z_n - w_m \vb \le \sqrt[k_n]{r_0}$ and for each $m \in D_a$ there exists exactly one $n \in C_a$ with $\va w_m - z_n \vb \le \sqrt[l_m]{s_0}$.
 \end{claim}

For $n \in C_0$, we define $U_n$ as  the set of all $m \in D_0$ for which  $w_m \in D^+ \pl z_n, \sqrt[k_n]{r_0} \pr$. Also, for $m \in D_0$, $V_m$ is the set of all $n \in C_0$ such that $z_n \in  D^+ \pl w_m , \sqrt[l_m]{s_0} \pr$. By Lemma~\ref{tus},   the set $C_1$ of all $n \in C_0$ with nonempty $U_n$  belongs to $\ef$. In a similar way we define $D_1$ as the set of all $m \in D_0$ with nonempty $V_m$, which belongs to $\af$.

Suppose first that $C_1' := \tl n \in C_1 : \ca U_n =1 \tr \in \ef$  and that $D_1' := \tl m \in D_1 : \ca V_m =1 \tr \in \af$. Then, by Lemma~\ref{tus}, $C_1'' := C_1' \cap B_{D_1'}^{s_0} \in \ef$  and $D_1'' := D_1' \cap B_{C_1'}^{r_0} \in \af$. Let $n \in C_1''$. Since $n \in B_{D_1'}^{s_0}$, there exists (a unique) $m \in D_1'$ with $\va z_n -  w_m \vb \le \sqrt[l_m]{s_0} $. Also, as $n \in C_1'$,   there exists a unique $m' \in D_0$ with $\va z_n -  w_{m'} \vb \le \sqrt[k_n]{r_0} $. Obviously,  if $\sqrt[k_n]{r_0} \le \sqrt[l_m]{s_0}$, then $\va w_{m'} - w_m \vb \le \sqrt[l_m]{s_0}$, and consequently  $m' = m$.
On the other hand, if $\sqrt[l_m]{s_0} < \sqrt[k_n]{r_0} $, then $m, m' \in U_n$ and, since $n \in C_1$, we also obtain $m=m'$.
This implies that there exists just  one $m \in D_1'$ with $\va z_n -  w_m \vb \le \sqrt[k_n]{r_0} $, and it is immediate that $m \in D_1''$.

Now, it is straightforward to check that, under the above assumptions, the claim is satisfied for $C_a := C_1''$ and $D_a := D_1''$. We next prove it when  $C_1' \notin \ef$ or $D_1' \notin \af$.

We assume without loss of generality that $C_1' \notin \ef$,  that is, the set $C_2$ of all $n \in C_1$ for which $\ca U_n \ge 2$ belongs to $\ef$. Clearly, if $n \in C_2$ and  $U_n = \tl m_1, \ldots, m_k \tr$, then, since $U_n \subset D_0$,  $\sqrt[l_{m_i}]{s_0} < \sqrt[k_n]{r_0}$ for all $i \in \tl 1, \ldots, k \tr$. Consider the set $D_2 :=  D_1 \cap B_{C_2}^{r_0}$,
which   belongs to $\af$, and fix $m \in D_2$. Since $m \in  B_{C_2}^{r_0}$,   there exists (a unique) $n_m \in C_2$ with $  m \in U_{n_m}$. Since $m \in D_1$, $V_m $ is nonempty, and there exists $n'_m \in C_0$ with $\va z_{n'_m} - w_m \vb \le \sqrt[l_m]{s_0}$. Taking into account that $\sqrt[l_m]{s_0} < \sqrt[k_{n_m}]{r_0}$, we see  that $\va z_{n'_m} - z_n \vb \le \sqrt[k_{n_m}]{r_0}$ and, by definition of $C_0$, $z_{n'_m} = z_{n_m}$. This shows in particular that  $V_m$ consists of just one point $n_m$ for all $m \in D_2$, and that this point $n_m$ belongs to $C_2$. This is to say that $\ca U_n \ge 2$ for all $n \in C_0 \cap B_{D_2}^{s_0}$. Also, since for each $n \in C_0 \cap B_{D_2}^{s_0}$  there is a unique $m_n \in D_2$ with $\va z_n - w_{m_n} \vb \le \sqrt[l_{m_n}]{s_0}$, we deduce that $\sqrt[l_{m_n}]{s_0}  < \sqrt[k_n]{r_0}$ and $m_n \in   U_n$. On the other hand, if there exists $m \in D_2 \cap U_n$ with  $m  \neq m_n$, then
 $\va z_n - w_m \vb \le \sqrt[k_n]{r_0}$ and we necessarily have $\sqrt[l_m]{s_0}  < \sqrt[k_n]{r_0}$. This implies that $\va z_n - z_{n_m} \vb \le \sqrt[k_n]{r_0}$, which is impossible. We conclude that $D_2 \cap U_n = \tl m_n \tr$ for all $n \in C_0 \cap B_{D_2}^{s_0}$. It is easy to see that, if we define  $C_a :=  C_0 \cap B_{D_2}^{s_0}$ and $D_a := D_2$, then the claim follows.
$\blacksquare$

\smallskip

We use Claim~\ref{arxona} to define a bijective map $\mathbf{j} : D_a   \rightarrow    C_a $ in such a way that, by Lemma~\ref{tus}, for $D \subset D_a$, $\mathbf{j} (D) \in \ef$ if and only if $D \in \af$.
We next study $L: = \lim_{\af} l_m / k_{\mathbf{j} \pl m \pr}$. We first suppose that $L = + \infty$. If $s \in (0, s_0)$, then we can write $s = {r_0}^t$ for some $t \in (0, + \infty)$, and there exists $D_s \in \af$ with
$D_s \subset D_a$ such that $l_m / t \ge k_{\mathbf{j} (m)}$ for all $m \in D_s$. This implies that $\sqrt[k_{\mathbf{j} (m)}]{r_0} \le \sqrt[l_m]{s}$ for all $m \in D_s$. It is easy to conclude that $\zeta_{\ze, \ef}^{\kai, r_0} \le \zeta_{\we, \af}^{\eli, 0} $ and, by
Corollary~\ref{lz}, $\ker \zeta_{\we, \af}^{\eli, 0} \subsetneq  \ker \zeta_{\ze, \ef}^{\kai,0 }  $, against our hypothesis. On the other hand, it is straightforward to see that, if $L = 0$, then $\lim_{\ef} k_n / l_{\mathbf{j}^{-1}  (n)} = + \infty  $, which is also  impossible, as above. We deduce that $0 < L < + \infty$. Now, again by Lemma~\ref{tus}, given $r \in (0, r_0)$ and $C \in \ef$ with $C \subset C_a$, the set $D_a \cap B_C^r$ belongs to $\af$, and it is clear that $D^+ \pl w_m , \sqrt[k_{\mathbf{j} (m)}]{r} \pr = D^+ \pl z_{\mathbf{j} (m)} , \sqrt[k_{\mathbf{j} (m)}]{r} \pr$ for all $m \in D_a \cap B_C^r$. By Corollary~\ref{pernia}, this means that $\zeta_{\ze, \ef}^{\kai, r} = \zeta_{\we, \af}^{\eli, r^L}$, and the conclusion follows easily.

\smallskip

{\em Proof of the first part.}
 Consider a sequence $\we  $ in $\de$, $\af \in \beta \en \setminus \en$, $ \eli   \in
 \com_{\af}  (\we)$,
 and $s_0 \in (0,1)$, and suppose that $\ker \zeta^{\kai, 0}_{\ze, \ef} = \ker \zeta^{\eli, s_0}_{\we, \af}$. Taking into account Corollaries~\ref{lluc} and~\ref{lucia}, we necessarily have
  $\lim_{\ef} k_n = + \infty$.

We can assume that  $M := \inf_{i \in \en} \prod_{j \neq i} \va z_i - z_j \vb^{k_j} >0$ and, by Remark~\ref{honble}, take $r_0 \in (0,M)$ so  that
 all the  disks $D^+ \pl z_i , \sqrt[k_i]{r_0} \pr$ are pairwise disjoint. Assume also that $D^+ \pl z_i , \sqrt[k_i]{r_0} \pr \subset C \pl 0, \va z_i \vb \pr$ for all $i$.

We next introduce some notation. Given a set $D \subset \en$, for $n \in \en$ and $r \in (0, r_0]$  we put
$$\Lambda_n^r (D) := \sum_{\substack{\va z_n - w_m \vb \le \sqrt[k_n]{r} \\ m \in D}} l_m , $$
$$\Upsilon_n^r (D) := \prod_{\substack{ \sqrt[k_n]{r} < \va z_n - w_m \vb \le \sqrt[k_n]{r_0} \\ m \in D} } \va z_n - w_m \vb^{l_m} , $$
and
$$\Sigma_n^r (D) := \prod_{\substack{    \va z_n - w_m \vb \le \sqrt[k_n]{r} \\ m \in D} } \va z_n - w_m \vb^{l_m} . $$
Then, we define
$\Upsilon^r (D) := \lim_{\ef} \Upsilon_n^r (D)$ and
$\Sigma^r (D) := \lim_{\ef} \Sigma_n^r (D)$. Finally, we set
$\Upsilon (D) := \lim_{r \ra 0} \Upsilon^r (D)$ and $\Sigma (D) := \lim_{r \ra 0} \Sigma^r (D)$.

We fix a strictly decreasing  sequence $(\alpha_i)$ in $(0,r_0]$ converging to $0$.

\begin{claim}\label{mcv}
If $\beta_{\alpha_i} := \lim_{\ef} \Lambda_n^{\alpha_i} \pl B_{\en}^{r_0} \pr / k_n $ for all $i$,  then
 $\lim_{i \ra \infty} \beta_{\alpha_i} =0$.
 \end{claim}
Consider
  $f \in \hac$ having exactly $l_m$ zeros close enough to each $w_m$ for $m \in B_{\en}^{r_0} $ as given in Proposition~\ref{gatillo}, and no other zeros in the circles $C \pl 0, \va  w_m \vb \pr$. Then, as in the proof of Proposition~\ref{diegosoper}, $ \zeta_{\we, \af}^{\eli, s_0} (f) >0$, so $\alpha := \zeta^{\kai, 0}_{\ze, \ef} (f) >0$ and
$\zeta_{\ze, \ef}^{\kai, r} (f) \ge \alpha$ for every $r \in (0,r_0]$.  By Corollary~\ref{conillosis},
for all  $r \in (0, r_0]$,
\begin{equation}\label{picazores}
 \alpha  \le \vc f \vd \lim_{\ef} \xi_{D^+ \pl z_n , \sqrt[k_n]{r} \pr} (f) \le \vc f \vd \lim_{\ef} r^{\Lambda_n^r \pl B_{\en}^{r_0} \pr / k_n} \Upsilon^r \pl B_{\en}^{r_0} \pr  .
\end{equation}

 If we now suppose that there exists $K >0$ such that
 $\beta_{\alpha_i}   \ge K $
  for infinitely many $i$, then Inequality~\ref{picazores} gives
$$\alpha \le \vc f \vd \lim_{\ef}    {\alpha_i}^{\Lambda_n^{\alpha_i} \pl B_{\en}^{r_0}   \pr /k_n} \le  \vc f \vd {\alpha_i}^K, $$ which is absurd.
$\blacksquare$

\smallskip

 Note   that, by Inequality~\ref{picazores}, $\Upsilon \pl B_{\en}^{r_0} \pr  >0$,
 and that, for  each $i \in \en$,
$$ C_i := \tl n \in \en : \Upsilon_n^{\alpha_i} \pl B_{\en}^{r_0} \pr > \frac{i}{i+1} \Upsilon  \pl B_{\en}^{r_0} \pr \tr  $$
belongs to $\ef$. On the other hand,   we define
$$E_i := \tl  m \in B_{C_i}^{r_0} :  \va w_m - z_n \vb > \sqrt[k_n]{\alpha_i} \  \forall n \in \en \tr  $$
and  $G_0 := \bigcup_{i =1}^{\infty} E_i $.

\begin{claim}\label{borja}
 $\af $ belongs to $ \cn G_0$.
\end{claim}
Suppose to the contrary that $\af \notin \cn G_0$.  Then we find a closed and open neighborhood $U \subset \beta \en$ of $\af$
with $U \subset B_{\en}^{r_0}$,
such that   $U \cap E_i = \emptyset$ for all $i \in \en$.
 This
implies that, for each $i$,
if $m \in U$ satisfies
$\va w_m - z_n \vb \le \sqrt[k_n]{r_0} $ for some
$n \in C_i$, then $w_m$ belongs to
$D^+ \pl z_n, \sqrt[k_n]{\alpha_i} \pr $. Therefore, for  each  $i \in \en$ and
$n \in C_i$,  $\Lambda_n^{r_0} (U)    \le \Lambda_n^{\alpha_i} \pl B_{\en}^{r_0} \pr $,
and  consequently  $$\lim_{\ef} \frac{\Lambda_n^{r_0} (U) }{k_n} \le \beta_{\alpha_i}$$ for all $i$. By Claim~\ref{mcv},  $\lim_{\ef} \Lambda_n^{r_0} (U)  / k_n =0$, so there exists  $(N_n)$
with $\lim_{\ef} N_n = + \infty$ and $\lim_{\ef} N_n \Lambda_n^{r_0} (U)  / k_n =0$. We select $C_0 \in \ef$ with $N_n \Lambda_n^{r_0} (U) \le k_n$ for all $n \in C_0$. Now, consider a sequence $(\delta_m)$ of positive numbers with  $\lim_{m \ra \infty} \delta_m =0$, in such a way that the disks $  D^+ \pl w_m , \delta_m   \pr$ are pairwise disjoint and $\delta_m \le \sqrt[l_m]{s_0}$ for all $m \in \en$. Then take a function $g \in \hac$ having exactly $N_n l_m$ zeros at distance less than $\delta_m$ to each $w_m$ with $\va w_m - z_n \vb \le \sqrt[k_n]{r_0}$ and $m  \in U$, for $n \in C_0$, and no other zeros in those circles (see Proposition~\ref{gatillo}). It is straightforward to see that, for all $r \in (0, r_0]$ and $n \in C_0$,
$$\xi_{D^+ \pl z_n , \sqrt[k_n]{r} \pr} (g) \ge r^{N_n \Lambda_n^{r_0} (U)  / k_n} \prod_{\substack{\va  z_j \vb = \va z_n \vb \\ j \in C_0 \\ j \neq n}} \va z_n - z_j \vb^{N_j \Lambda_j^{r_0} (U)} \ge r^{N_n \Lambda_n^{r_0} (U) / k_n } M.$$
Taking into account Corollary~\ref{conillosis}, we see that $\zeta_{\ze, \ef}^{\kai, r} (g) \ge M \vc g \vd$ for all $r$, so
$\zeta_{\ze, \ef}^{\kai, 0} (g) >0$. On the other hand, for each $K \in \en$, the set $D_K := \tl n \in C_0 : N_n \ge K \tr$ belongs to $\ef$ and, by Lemma~\ref{tus}, $B_{D_K}^{r_0}$ belongs to $\af$. Also, for
$m \in B_{D_K}^{r_0} \cap U$,  $\xi_{D^+ \pl w_m , \sqrt[l_m]{s_0} \pr} (g) \le {s_0}^K$, so
$\zeta^{\eli, s_0}_{\we, \af} (g)  = 0$, which is impossible. This shows that $\af$ belongs to $\cn G_0$. $\blacksquare$

\smallskip

\begin{claim}\label{roberto}
$\Upsilon (G_0) >0$ and $\Sigma  (G_0) =1$.
\end{claim}

It is clear that $\Upsilon (G_0)   \ge \Upsilon \pl  B_{\en}^{r_0} \pr >0$. We first see that, in fact,  $\Upsilon (G_0)   = \Upsilon \pl  B_{\en}^{r_0} \pr$. Note that, by definition,
$$E_i  = \bigcup_{n \in C_i}  \tl  m \in B_{\en}^{r_0} : \sqrt[k_n]{\alpha_i} < \va z_n - w_m \vb \le \sqrt[k_n]{r_0} \tr   . $$
This implies that, if for  $n \in C_i$ and $A  $ with $E_i \subset A \subset B_{\en}^{r_0} $  we denote
 $H_n^i (A) :=
 \tl   m \in A : \sqrt[k_n]{\alpha_i} < \va z_n - w_m \vb \le \sqrt[k_n]{r_0}    \tr $,
 then  $H_n^i (E_i) =  H_n^i \pl B_{\en}^{r_0} \pr =  H_n^i (A)$. In particular, $H_n^i (G_0) = H_n^i \pl B_{\en}^{r_0}  \pr$, and we easily deduce that $\Upsilon_n^{\alpha_i} (G_0) = \Upsilon_n^{\alpha_i} \pl B_{\en}^{r_0}  \pr $ whenever $n \in C_i$.
 Thus, $\Upsilon^{\alpha_i} (G_0)  = \Upsilon^{\alpha_i} \pl  B_{\en}^{r_0} \pr$ for all $i$, and $\Upsilon (G_0)   = \Upsilon \pl  B_{\en}^{r_0} \pr$.

 Next, for  $n \in C_1$ fixed, put $T_n := \min \tl \Upsilon_n^{\alpha_i} \pl B_{\en}^{r_0} \pr : n \in C_i \tr$  and $I(n) := \min \tl i \in \en : T_n = \Upsilon_n^{\alpha_i} \pl B_{\en}^{r_0} \pr  \tr$.
  We see that, if $i > I(n)$, then either $n \notin C_i$ or  $\Upsilon_n^{\alpha_i} \pl  B_{\en}^{r_0} \pr = T_n  $. In the first case, if $m \in B_{C_i}^{r_0}$, then $\va z_n - w_m \vb > \sqrt[k_n]{r_0}  $, so $$ D^+ \pl z_n, \sqrt[k_n]{r_0} \pr \cap \tl w_m : m \in E_i \tr = \emptyset.$$
 In the second case, the set $$H_n^i \pl B_{\en}^{r_0} \pr \setminus H_n^{I(n)} \pl B_{\en}^{r_0} \pr = \tl  m : \sqrt[k_n]{\alpha_i} < \va z_n - w_m \vb \le \sqrt[k_n]{\alpha_{I(n)}} \tr$$ is empty, so $H_n^i \pl B_{C_i}^{r_0} \pr \subset H_n^{I(n)} \pl B_{\en}^{r_0} \pr$.
 Thus,   the set
\begin{eqnarray*}
  G_n &:=& \tl m \in G_0 : \va z_n - w_m \vb \le \sqrt[k_n]{r_0} \tr \\
  &=& \bigcup_{i=1}^{\infty} \tl m \in E_i : \va z_n - w_m \vb \le \sqrt[k_n]{r_0} \tr \\
  &=& \bigcup_{i=1}^{\infty} H_n^i \pl B_{C_i}^{r_0} \pr
   \end{eqnarray*}
   is contained in $H_n^{I(n)} \pl B_{\en}^{r_0} \pr$,
   and consequently
 $$\Sigma_n^{r_0} (G_0) = \prod_{  m \in G_n  } \va z_n - w_m \vb^{l_m} \ge \Upsilon_n^{\alpha_{I(n)}} \pl B_{\en}^{r_0} \pr . $$
We deduce that,  since $n \in C_{I(n)}$ and $\Upsilon (G_0) = \Upsilon \pl  B_{\en}^{r_0} \pr$, $$\Sigma_n^{r_0} (G_0) >  (I(n) / I(n) +1)  \Upsilon (G_0). $$

Now fix $N \in \en$. We next see that the set $A (N)$ of all $n \in C_1$ such that $I(n) \ge N$ belongs to $\ef$.
If this is not the case, then for an $i_0 < N$, the set $L $ of all $n \in C_1$ such that $I (n) = i_0$ belongs to $\ef$, and by Lemma~\ref{tus}  the set $B_L^{r_0}$ belongs to $\af$. Also, by Claim~\ref{borja}, $G_0$ belongs  to $\af$ (see \cite[p. 90]{GJ}), and so does  $G_0 \cap B_L^{r_0}   $.
On the other hand, as above, if $n \in L$, then $G_n \subset  H_n^{i_0} \pl B_{\en}^{r_0} \pr \subset B_{\en}^{r_0} \setminus B_{\en}^{\alpha_{i_0}} $.
Consequently  $G_0 \cap B_L^{r_0}  = \bigcup_{n \in L} G_n \subset B_{\en}^{r_0} \setminus B_{\en}^{\alpha_{i_0}}$, and we conclude that $B_{\en}^{\alpha_{i_0}}$ does not belong to $\af$, against Lemma~\ref{tus}.

 It is clear that, for $n \in A (N)$, $$\Sigma_n^{r_0} (G_0) >  (I(n) / I(n) +1)  \Upsilon (G_0) \ge  (N / N+1)  \Upsilon (G_0) .$$
We easily deduce that $\Sigma^{r_0} (G_0) \ge \Upsilon (G_0)$. On the other hand, it is straightforward to see that, for all $i \in \en$,
   $\Sigma^{r_0} (G_0) = \Sigma^{\alpha_i} (G_0) \Upsilon^{\alpha_i} (G_0) $, so $\Sigma^{r_0} (G_0) = \Sigma  (G_0) \Upsilon (G_0)$.  Now it is immediate    that $\Sigma  (G_0) =1$. $\blacksquare$

   \smallskip

By Claim~\ref{mcv}, there exists a strictly decreasing sequence $(r_i)_{i \ge 2}$ in $(0,1)$, convergent to $0$, with $$ \lim_{\ef} \frac{\Lambda_n^{r_i} \pl G_0 \pr}{k_n} < \frac{1}{2^i}$$
for all $i \ge 2$. Also, let
$\pl R_i \pr_{i \ge 2}$ be a strictly increasing sequence in $(0,1)$ with $R:= \prod_{i=2}^{\infty} {R_i}^i >0$.
By Claim~\ref{roberto}, and  taking a subsequence if necessary, we assume that
$\Sigma^{r_i}(G_0 ) > R_i$ and $\Upsilon (G_0 ) / \Upsilon^{r_i} (G_0 ) > R_i$ for all $i \ge 2$.
 We put $r_1 := r_0$.

 Note that, given $i \ge 2$, the set
$$D_i :=   \tl n \in \en: \Sigma_n^{r_j} (G_0) > R_j \mbox{ and } 2^j\Lambda_n^{r_j} \pl G_0 \pr < k_n \mbox{ whenever } 2 \le j \le i \tr $$ belongs to $\ef$,
and   by Lemma~\ref{tus}, $B_{  D_i}^{r_i}$ is in $\af$. Also, if $$D'_i := \tl n \in D_i : D^+ \pl z_n , \sqrt[k_n]{r_i} \pr \cap \tl w_m : m \in G_0 \tr \neq \emptyset \tr$$ does not belong to $\ef$, then, again by Lemma~\ref{tus}, $B_{D_i \setminus D'_i} $ belongs to $\af$, and so does $G_0 \cap B_{D_i \setminus D'_i}^{r_i}  = \emptyset $. We conclude that each $D'_i$ is in $\ef$. Note also that $D'_2 \supset D'_3 \supset \cdots$ and that, by construction, $z_n \neq w_m$ for all $m \in G_0$ and $n \in \en$, so $\sup \tl i : n  \in D'_i \tr     \in \en$ for all $n$.

 Now, for $j$ and $n$ in $ \en$, put $P_j^n := \tl m \in G_0 :     \sqrt[k_n]{r_{j+1}} < \va z_n - w_m \vb \le \sqrt[k_n]{r_j}  \tr$.
It is easy to see that the sets $P_j^n$ are pairwise disjoint and $G_0 = \bigcup_{n, j} P_j^n$, so
given $m \in G_0$, there exist  unique
$j_m, n_m \in \en  $ with $m \in P_{j_m}^{n_m}$.

We define, when $n_m \in D'_2$ and $j_m \ge 2$,
$$d_m := \min \tl  j_m,  \max \tl i : n_m \in D'_i \tr   \tr .$$
Since $d_m \ge i$ for all $m \in G_0 \cap B_{  D'_i}^{r_i}$,  $\lim_{\af} d_m = + \infty$.

\begin{claim}\label{isere}
Let $i \ge 2$. For $ n \in D'_i$,
$$ {\Sigma'}_n^{r_i} (G_0):= \prod_{\substack{\va z_n - w_m   \vb \le \sqrt[k_n]{r_i} \\ m \in G_0}} \va z_n - w_m \vb^{d_m l_m} \ge R.$$
\end{claim}

Let $j_0 :=    \max \tl j : n \in D'_j    \tr $. If $q \ge j_0$, then $d_m = j_0$ for all $m \in P_q^n$.
Also, if $i <  j_0 $ and $i \le q \le j_0 -1$, then
$d_m = q$ for all $m \in P_q^n$.
Consequently,
$$
{\Sigma'}_n^{r_i} (G_0)  = \prod_{q = i}^{j_0 -1} \prod_{\substack{      m \in P_q^n }} \va z_n - w_m \vb^{q l_m}  \cdot \prod_{\substack{\va w_m - z_n \vb \le \sqrt[k_n]{r_{j_0 }} \\ m \in G_0}} \va z_n - w_m \vb^{j_0 l_m}
  ,
$$
implying that ${\Sigma'}_n^{r_i} (G_0) \ge  \prod_{q=i}^{j_0 }    {R_q}^{q} \ge R$.
$\blacksquare$

\smallskip

Now we
extend the definition of $d_m$ to all $G_0$ by putting $d_m :=0$ for all $m \in G_0 \setminus B_{D'_2}^{r_2}$, and
consider $g \in \hac$ having $d_m l_m$ zeros close enough to each $w_m$ for $m \in G_0  $ and no other zeros in the circles containing those $w_m$ (see Proposition~\ref{gatillo}).
It is straightforward to see that, since $\lim_{\af} d_m = + \infty$,   $\zeta^{\eli, s_0}_{\we, \af}  (g) =0$.
Next, we prove that
$\zeta^{\kai, 0}_{\ze, \ef}  (g) > 0$, against our assumption.
We write, for $   i \in \en$ and $ n \in D'_2$,
$${\Lambda'}_n^{r_i} (G_0) := \sum_{\substack{\va z_n - w_m \vb \le \sqrt[k_n]{r_i} \\ m \in G_0}} d_m l_m . $$

\begin{claim}\label{nadienlafaku}
For all $n \in D'_2$, ${\Lambda'}_n^{r_1} (G_0) \le 2 k_n  $.
\end{claim}
For $n \in D'_2$ fixed, let $j_0 := \max \tl j : n \in D'_j    \tr$. Note first that ${\Lambda'}_n^{r_1} (G_0) = {\Lambda'}_n^{r_2} (G_0)$ because $d_m =0$ for all  $m \notin B_{D'_2}^{r_2}$.
Also, taking into account that, if $j \ge 2$ and $m \in P_j^n$,  then $d_m \le j_0$,
we have
\begin{eqnarray*}
{\Lambda'}_n^{r_2} (G_0)
&=& \sum_{ i=2 }^{j_0} \sum_{\substack{\va z_n - w_m \vb \le \sqrt[k_n]{r_2} \\   d_m =i }} i l_m \\
&=&   \sum_{\substack{\va z_n - w_m \vb \le \sqrt[k_n]{r_2} \\    d_m \ge 2 }} 2 l_m + \sum_{ i=3 }^{j_0} \sum_{\substack{\va z_n - w_m \vb \le \sqrt[k_n]{r_2}  \\   d_m \ge i }}   l_m , \\ \
&=&   \sum_{\substack{\va z_n - w_m \vb \le \sqrt[k_n]{r_2} \\    m \in G_0 }} 2 l_m + \sum_{ i=3 }^{j_0} \sum_{\substack{\va z_n - w_m \vb \le \sqrt[k_n]{r_i}  \\   m \in G_0 }}   l_m ,
\end{eqnarray*}
that is,
${\Lambda'}_n^{r_2} (G_0) = 2 \Lambda_n^{r_2} (G_0) + \sum_{ i=3 }^{j_0}  \Lambda_n^{r_i} (G_0) $.
 Since  $2^i \Lambda_n^{r_i} (G_0) < k_n$ for $i \ge 2$ with $n \in D'_i$, the claim follows.
$\blacksquare$

\smallskip

For $i \ge 3$ and $n \in D'_i$,  let
$$
{\Upsilon'}_n^{r_i}  (G_0) := \prod_{j=2}^{i-1} \pl \prod_{ m \in P_j^n} \va z_n - w_m \vb^{l_m} \pr^j  .
$$

\begin{claim}\label{hombro}
${\Upsilon'}^{r_i} (G_0) :=
\lim_{\ef} {\Upsilon'}_n^{r_i}  (G_0) \ge R $ for all $i$.
\end{claim}

First  define,
 for $j \in \en$, $W_j^n:=  \prod_{ m \in P_j^n} \va z_n - w_m \vb^{l_m}$ and $W_j := \lim_{\ef} W_j^n$. Clearly, for $j \ge 2$, $W_j^n  = \Upsilon^{r_{j+1}}_n (G_0) / \Upsilon^{r_j}_n (G_0)$  and       $$W_j  =   \Upsilon^{r_{j+1}} (G_0) / \Upsilon^{r_j} (G_0) \ge \Upsilon (G_0) / \Upsilon^{r_j} (G_0)  > R_j .$$
 Now, ${\Upsilon'}^{r_i} (G_0)  =   \prod_{j=2}^{i-1} \pl W_j \pr^j \ge R $,
and we are done.
$\blacksquare$

\smallskip

Fix $i \ge 3$.  Note that,
if  $n \in D'_i$ and $m \in P_j^n$ for $2 \le j \le i$, then $d_m = j$, so
\begin{equation}\label{rezero}
\xi_{D^+ \pl z_n , \sqrt[k_n]{r_i} \pr} (g)  = {r_i}^{{\Lambda'}_n^{r_i} (G_0) /k_n} {\Upsilon'}_n^{r_i}  (G_0)  \prod_{\substack{\va z_j \vb = \va z_n \vb \\ j \in D_2' \\ j \neq n}} \va z_n - z_j \vb^{{\Lambda'}_j^{r_1} (G_0)}. \end{equation}
On the other hand, taking into account Claim~\ref{isere}, for $n \in  D'_i$,
\begin{equation}\label{rezzero}
 {r_i}^{{\Lambda'}_n^{r_i} (G_0) /k_n} \ge {\Sigma'}_n^{r_i} (G_0) \ge R .
\end{equation}

By Claim~\ref{nadienlafaku},
${\Lambda'}_j^{r_1} (G_0) \le 2 k_j$, so  $\prod_{\substack{ j \in D_2' \\ j \neq n}} \va z_n - z_j \vb^{{\Lambda'}_j^{r_1} (G_0)} \ge M^2$.
By
Inequalities~\ref{rezzero} and~\ref{rezero}, and taking into account Claim~\ref{hombro} and Corollary~\ref{conillosis},
 $      \zeta^{\kai, r_i}_{\ze, \ef}  (g) \ge M^2 R^2 \vc g \vd  $.
Therefore $\zeta^{\kai, 0}_{\ze, \ef}  (g) \ge M^2 R^2 \vc g \vd >0$.

Consequently $\ker \zeta^{\kai, 0}_{\ze, \ef}  \neq \ker \zeta^{\eli, s_0 }_{\we, \af}$, and we are done.
\end{proof}

We finally see that Theorem~\ref{zerolari}  cannot be extended in general to other seminorms obtained as the infimum of a decreasing family.

\begin{ex}\label{nicolas}
Let $(z_n)$ be a sequence in $\de$ such that $\lim_{n \ra \infty} \va z_n \vb =1$ and  $(\va z_n \vb)$ is strictly increasing.  Let $\ef$ be a nonprincipal ultrafilter in $\en$, and suppose that   $\kai \in \com_{\ef} (\ze)$ satisfies  $\lim_{n \ra \infty} k_n = + \infty$.    As in the proof of Theorem~\ref{kubzdelamagdalena}, the family $F$ of all seminorms $\zeta_{\ze, \ef}^{\emi, r}$ such that $\lim_{\ef} m_n = + \infty$, $\lim_{\ef} m_n / k_n < + \infty$ and $r \in (0,1)$ is totally ordered. As a consequence, $\varphi (f) := \inf_{\phi \in F}  \phi (f)$, $f \in \hac$, defines an element in $\ma$. It is obvious that $\zeta_{\ze, \ef}^{\mathbf{1}. 1} \le \varphi$. Following the same idea as in the proof of Theorem~\ref{liberban}, it is   straightforward to see that, given $f \in \hac$ and $\epsilon >0$, there exists $\zeta_{\ze, \ef}^{\emi, r} \in F$ with $\zeta_{\ze, \ef}^{\emi, r} (f) < \zeta_{\ze, \ef}^{\mathbf{1}. 1} (f) + \epsilon$. We conclude that $\varphi = \zeta_{\ze, \ef}^{\mathbf{1}. 1}$.
\end{ex}

\bigskip

We end the paper by listing some questions for which we do not have an answer.
\begin{enumerate} 
\item[1.] Does there exist $\varphi \in \mil$ such that $\ker \psi \neq \ker \varphi$ for all $\psi \in \mim$? Does $\varphi = \ker \zeta_{\ze, \ef}^{\kai, 0}$ satisfy this property when $\kai \in \com_{\ef} (\ze)$ and $\lim_{\ef} k_n = + \infty$?
\item[2.] More generally, does there exist $\varphi \in \ma$ with {\em unique}  nonmaximal  kernel, that is, such that $\ker \psi \neq \ker \varphi$ whenever $\psi \in \ma$ and $\psi \neq \varphi$?
\item[3.] Does there exist $\varphi \in \ma$ with {\em nonmaximal} kernel such that $f \in \ker \varphi$ and $f' \notin \ker \varphi$ for some $f \in \hac$?
\item[4.] Does there exist $\varphi \in \ma$ with {\em maximal} kernel such that $f' \in \ker \varphi$ whenever $f  \in \ker \varphi$? (stated  in \cite{vdP})
\end{enumerate}

\end{document}